\documentclass[a4paper,12pt]{amsart}

\usepackage{amscd,amssymb,amsmath,amsfonts,amsthm, mathrsfs, graphicx,verbatim,hyperref,microtype,enumitem}
\usepackage[OT1,T1]{fontenc}
\usepackage[all]{xypic}
\usepackage{pdfpages}
\usepackage{tikz}
\usepackage[all]{xy}

\newtheorem{theorem}{Theorem}[section]
\newtheorem{corollary}[theorem]{Corollary}

\newtheorem{proposition}[theorem]{Proposition}
\theoremstyle{definition}
\newtheorem{remark}[theorem]{Remark}

\numberwithin{equation}{section}

\def\CA{{\mathscr A}}
\def\CB{{\mathscr B}}
\def\CK{{\mathscr K}}
\def\CV{{\mathscr V}}

\def\tee{t}
\def\eye{i}

\newcommand{\cd}[2][]{\vcenter{\hbox{\xymatrix#1{#2}}}}

\newcommand{\ltwocell    }[3][0.5]{\ar@{}[#2] \ar@{=>}?(#1)+/r 0.2cm/;?(#1)+/l 0.2cm/^{#3}}
\newcommand{\rtwocell    }[3][0.5]{\ar@{}[#2] \ar@{=>}?(#1)+/l 0.2cm/;?(#1)+/r 0.2cm/^{#3}}
\newcommand{\dtwocell    }[3][0.5]{\ar@{}[#2] \ar@{=>}?(#1)+/u 0.2cm/;?(#1)+/d 0.2cm/^{#3}}
\newcommand{\utwocell    }[3][0.5]{\ar@{}[#2] \ar@{=>}?(#1)+/d 0.2cm/;?(#1)+/u 0.2cm/_{#3}}

\newcommand{\dltwocell   }[3][0.5]{\ar@{}[#2] \ar@{=>}?(#1)+/ur 0.2cm/;?(#1)+/dl 0.2cm/^{#3}}
\newcommand{\drtwocell   }[3][0.5]{\ar@{}[#2] \ar@{=>}?(#1)+/ul 0.2cm/;?(#1)+/dr 0.2cm/^{#3}}
\newcommand{\ultwocell   }[3][0.5]{\ar@{}[#2] \ar@{=>}?(#1)+/dr 0.2cm/;?(#1)+/ul 0.2cm/^{#3}}
\newcommand{\urtwocell   }[3][0.5]{\ar@{}[#2] \ar@{=>}?(#1)+/dl 0.2cm/;?(#1)+/ur 0.2cm/^{#3}}

\newcommand{\stringpent}{
\begin{gathered}
\vcenter{\hbox{\begin{tikzpicture}[y=0.8pt, x=0.8pt,yscale=-1, inner sep=0pt, outer sep=0pt, every text node part/.style={font=\scriptsize} ]
  \path[draw=black,line join=miter,line cap=butt,line width=0.650pt]
    (0.5761,996.5573) .. controls (25.8299,996.0522) and (35.7165,994.4461) ..
    node[above right=0.12cm,at start] {$\!\!(\cde  1)  1$}(48.0259,1000.9550);
  \path[draw=black,line join=miter,line cap=butt,line width=0.650pt]
    (0.5761,1020.4270) .. controls (32.7994,1020.6795) and (38.4294,1009.3612) ..
    node[above right=0.12cm,at start] {$\!\bce  1$}(48.0259,1003.0477);
  \path[draw=black,line join=miter,line cap=butt,line width=0.650pt]
    (49.8969,1000.5044) .. controls (58.2307,980.6554) and (151.0215,979.4784) ..
    node[above right=0.25cm,pos=0.4] {$\mathfrak a  1$}(166.1738,987.8122);
  \path[draw=black,line join=miter,line cap=butt,line width=0.650pt]
    (49.6135,1003.1578) .. controls (60.9777,1006.6933) and (72.9109,1017.9473) ..
    node[below left=0.05cm,pos=0.57]{$\bde  1$} (80.4466,1027.1123);
  \path[draw=black,line join=miter,line cap=butt,line width=0.650pt]
    (0.5761,1042.0698) .. controls (14.4657,1042.3223) and (75.2981,1040.0698) ..
    node[above right=0.12cm,at start] {$\!\abe$} (80.8539,1029.7127);
  \path[draw=black,line join=miter,line cap=butt,line width=0.650pt]
    (82.2825,1030.8845) .. controls (90.1821,1054.0933) and (173.6677,1051.1703)
    .. node[above left=0.12cm,at end] {$\ade$}(214.7867,1051.3122);
  \path[draw=black,line join=miter,line cap=butt,line width=0.650pt]
    (141.5590,1019.2256) .. controls (143.0743,1011.6494) and (166.0258,992.5776)
    .. node[below right=0.04cm] {$1  \mathfrak a$}(166.0258,992.5776);
  \path[draw=black,line join=miter,line cap=butt,line width=0.650pt]
    (141.6203,1021.3685) .. controls (141.6203,1021.3685) and (163.4571,1014.0062)
    .. node[above left=0.12cm,at end] {$1  (1  \abc)\!\!$}(214.7868,1014.7638);
  \path[draw=black,line join=miter,line cap=butt,line width=0.650pt]
    (141.4111,1023.0061) .. controls (141.4111,1023.0061) and (142.5493,1033.5038)
    .. node[above left=0.12cm,at end] {$1  \acd$\!}(214.7868,1033.7564);
  \path[draw=black,line join=miter,line cap=butt,line width=0.650pt]
    (167.5590,988.4831) .. controls (167.5590,988.4831) and (177.6142,980.4754) ..
    node[above left=0.12cm,at end] {$\mathfrak a$\!} (214.7868,979.9703);
  \path[draw=black,line join=miter,line cap=butt,line width=0.650pt]
    (167.9162,991.3403) .. controls (172.8729,995.3375) and (185.9479,996.5546) ..
    node[above left=0.12cm,at end] {$\mathfrak a$\!}(214.7867,996.5546);
  \path[draw=black,line join=miter,line cap=butt,line width=0.650pt]
    (82.5536,1028.3546) .. controls (96.1906,1037.3648) and (132.5520,1035.8496)
    .. node[below right=0.06cm,pos=0.42] {$1  \abd$} (139.8756,1022.0412);
  \path[draw=black,line join=miter,line cap=butt,line width=0.650pt]
    (82.7084,1025.8212) .. controls (84.1609,1008.5230) and (138.1030,990.5459) ..
    node[above left=0.07cm,pos=0.6] {$\mathfrak a$}(166.0786,990.2620)
    (49.8145,1001.8228) .. controls (60.2802,1001.0663) and (79.7259,1003.0374) ..
    node[above right=0.08cm,pos=0.4,rotate=-9] {$(1  \bcd)  1$}(98.0839,1006.4270)
    (105.9744,1007.9878) .. controls (121.2679,1011.2244) and (134.6866,1015.3648) ..
    node[below left=0.08cm,pos=0.7,rotate=-15] {$1  (\bcd  1)$}(139.8779,1019.5827);
  \path[fill=black] (166.83624,990.0094) node[circle, draw, line width=0.65pt,
    minimum width=5mm, fill=white, inner sep=0.25mm] (text3313) {$
      \pi$    };
  \path[fill=black] (48.501251,1001.512) node[circle, draw, line width=0.65pt,
    minimum width=5mm, fill=white, inner sep=0.25mm] (text3305) {$\bcde  1$    };
  \path[fill=black] (81.600487,1027.7435) node[circle, draw, line width=0.65pt,
    minimum width=5mm, fill=white, inner sep=0.25mm] (text3309) {$\abde$    };
  \path[fill=black] (145.38065,1021.0309) node[circle, draw, line width=0.65pt,
    minimum width=5mm, fill=white, inner sep=0.25mm] (text3317) {$1  \abcd$  };
\end{tikzpicture}}} \ = \
\vcenter{\hbox{\begin{tikzpicture}[y=0.8pt, x=0.7pt,yscale=-1, inner sep=0pt, outer sep=0pt, every text node part/.style={font=\scriptsize} ]
  \path[draw=black,line join=miter,line cap=butt,line width=0.650pt]
    (0.5761,1010.9221) .. controls (22.7994,1011.1746) and (33.5330,1012.5294) ..
    node[above right=0.12cm,at start] {\!$\bce  1$}(49.1780,1027.2420);
  \path[draw=black,line join=miter,line cap=butt,line width=0.650pt]
    (0.5761,1042.0698) .. controls (14.4657,1042.3223) and (43.6151,1040.3578) ..
    node[above right=0.12cm,at start] {\!$\abe$}(49.1709,1030.0007);
  \path[draw=black,line join=miter,line cap=butt,line width=0.650pt]
    (51.1755,1030.5965) .. controls (67.4280,1056.6856) and (126.7968,1041.6654)
    .. node[below right=0.14cm] {$\ace$} (139.8511,1033.1665);
  \path[draw=black,line join=miter,line cap=butt,line width=0.650pt]
    (142.8500,1030.5854) .. controls (142.8500,1030.5854) and (184.5612,1028.3914)
    .. node[above left=0.12cm,at end] {$1  \acd$\!}(225.8909,1029.1490);
  \path[draw=black,line join=miter,line cap=butt,line width=0.650pt]
    (142.6407,1033.9511) .. controls (142.6407,1033.9511) and (172.9964,1047.0248)
    .. node[above left=0.12cm,at end] {$\ade$\!}(225.8909,1046.9894);
  \path[draw=black,line join=miter,line cap=butt,line width=0.650pt]
    (51.8895,1026.1092) .. controls (64.3077,997.3630) and (142.1189,975.3128) ..
    node[above left=0.12cm,at end] {$\mathfrak a$\!}(225.8909,975.6050)
    (52.8868,1028.6427) .. controls (64.4224,1023.3508) and (87.4860,1019.3533) ..
    node[below right=0.1cm,pos=0.53] {$1  \abc$} (113.0660,1016.4776)
    (121.4833,1015.5795) .. controls (135.8648,1014.1237) and (150.7159,1013.0048) ..
    node[above=0.1cm] {$1  \abc$} (164.5150,1012.1936)
    (173.4370,1011.7035) .. controls (191.3809,1010.7870) and (206.7880,1010.4220) ..
    node[above left=0.12cm,at end] {$1  (1  \abc)\!\!$} (225.8909,1010.5363)
    (143.3647,1028.4425) .. controls (173.9707,1013.3776) and (183.0558,992.0340) ..
    node[above left=0.12cm,at end] {$\mathfrak a$\!}(225.8909,992.6100)
    (1.1521,981.0038) .. controls (30.9106,980.4472) and (58.1432,987.9204) ..
    node[above right=0.12cm,at start] {$\!\!(\cde  1)  1$}(81.9942,998.1170)
    (88.1242,1000.8286) .. controls (108.0234,1009.9228) and (125.3927,1020.6737) ..
    node[above right=0.1cm,pos=0.22] {$\cde  1$}
    (139.6961,1029.7577);
  \path[fill=black] (49.917496,1028.8956) node[circle, draw, line width=0.65pt,
    minimum width=5mm, fill=white, inner sep=0.25mm] (text3309) {$\abce$    };
  \path[fill=black] (139.03424,1031.976) node[circle, draw, line width=0.65pt,
    minimum width=5mm, fill=white, inner sep=0.25mm] (text3317) {$\acde$  };
\end{tikzpicture}}}
\end{gathered}
}

\begin{document}

\author{Mitchell Buckley, Richard Garner, Stephen Lack and Ross Street}
\title{Skew-monoidal categories and the Catalan simplicial set}
\date{\small{1 July 2013}}

\address{
Centre of Australian Category Theory,
Macquarie University NSW 2109,
Australia}

\email{mitchell.buckley@mq.edu.au, richard.garner@mq.edu.au, steve.lack@mq.edu.au, ross.street@mq.edu.au}


\thanks{The authors gratefully acknowledge the support of
Australian Research Council Discovery Grant DP130101969,
Buckley's Macquarie University Postgraduate Scholarship,
Garner's Australian Research Council Discovery Grant DP110102360, and
Lack's Australian Research Council Future Fellowship. }

\maketitle

\section{Introduction}

The $n$th Catalan number $C_n$, given explicitly by $\frac{1}{n+1}\binom{2n}{n}$, is well-known to be the answer
to many different counting problems. 
For example, it is the number of bracketings of an $(n+1)$-fold product.
Thus there are many families of sets, $\mathbb C_n$, indexed by the natural numbers, whose cardinalities are the
Catalan numbers; such families might then be called ``Catalan sets''. Stanley~\cite{StanleyEC2, StanleyAddendum} describes at least 205 such families. 
We shall show how to define functions between these sets, in such a way as to produce a simplicial set $\mathbb C$, which is the ``Catalan simplicial set'' of the title. 
This is done using what seems to be a new description of the Catalan sets, which relies heavily on the Boolean algebra $\mathbf 2$.

Simplicial sets are abstract, combinatorial models of spaces, most often used in homotopy theory. They also arise, however,
as models of higher-dimensional categories, and that is their main role in this paper: our primary goal is to show that the simplicial
set $\mathbb C$ encodes, in a precise sense, a particular categorical structure called a {\em skew-monoidal category}. We shall show
that a skew-monoidal category is the same thing as a simplicial map from $\mathbb C$ to another simplicial set $N\mathrm{Cat}$,
the nerve of the monoidal bicategory $\mathrm{Cat}$ of categories and functors.

The structure of skew-monoidal category was introduced recently by Szlach\'anyi~\cite{Szl2012} in his study of bialgebroids, which are themselves
an extension of the notion of quantum group.

Thus the work presented here lies at the interface of several 
mathematical disciplines:
\begin{itemize}[nolistsep]
\item[(a)] algebraic topology, because it involves simplicial sets and nerves;
\item[(b)] combinatorics, in the form of the Catalan numbers;
\item[(c)] quantum groups, through recent work on bialgebroids;
\item[(d)] logic, because of the distinguished role of the Boolean algebra $\mathbf{2}$; and
\item[(e)] category theory.
\end{itemize}

The natural generalisation of the notion of monoid from the level of
sets to categories is that of \emph{monoidal category}~\cite{CWM}. Like a
monoid, a monoidal category $\CA$ comes equipped with a a unit element
$I \in \CA$ and a multiplication $\otimes \colon \CA \times
\CA \to \CA$ (now a functor, rather than a
function); but unlike a monoid, the basic operations of a monoidal
category do not satisfy associativity and unitality on the nose, but
only up to coherent natural families of isomorphisms:
\begin{equation}\label{eq:constraints}
\begin{gathered}
  \lambda_A \colon I \otimes A \to A \qquad \text{and} \qquad 
  \rho_A \colon A \to A \otimes I \quad \text{(for $A \in \CA$)}\\
  \alpha_{ABC} \colon (A \otimes B) \otimes
  C \to A \otimes (B \otimes C)\quad \text{(for $A,B,C, \in
    \CA$).}
\end{gathered}
\end{equation}
In calling these families coherent, we mean to say that certain
diagrams of derived natural transformations, built from composites of
tensors of $\alpha$'s, $\lambda$'s and $\rho$'s, must commute. The
commutativity of these particular diagrams in fact implies the
commutativity of \emph{all} such diagrams: this is one form of the
coherence theorem for monoidal categories~\cite{CWM}, and justifies the
choice of axioms made. Mac\ Lane's original formulation~\cite{ML1963}
posited five generating axioms, recalled in Section~\ref{sec:skew}
below; Kelly later reduced these to two~\cite{Kelly1964}.

Recently, Szlach\'anyi~\cite{Szl2012} has introduced \emph{skew-monoidal
  categories}. The basic data for a skew-monoidal category are the
same as for a monoidal category, except that the constraint morphisms
in~\eqref{eq:constraints} are no longer required to be invertible. The
axioms which are to be satisfied are Mac\ Lane's original five, rather
than Kelly's two; the choice is substantive, as without invertibility
in~\eqref{eq:constraints}, the two axiomatisations are no longer
equivalent. 
Szlach\'anyi's motivation in~\cite{Szl2012} for defining skew-monoidal
structure comes from representation theory: he identifies left
bialgebroids based on a ring $R$ with closed skew-monoidal structures
on the category of left $R$-modules. Lack and Street have placed this
result in a more general context, showing in~\cite{Smswqc} that the quantum
categories of~\cite{QCat} can be captured as
\emph{skew monoidales}---internal skew-monoidal objects---in a
suitable monoidal bicategory.

Whilst these applications justify the use of skew-monoidal structure,
they do not give an intrinsic justification for the form the structure
takes. There are in fact two places in the definition where a
non-obvious choice has been made.  The first concerns the orientation
of the maps in~\eqref{eq:constraints}.  For example, had we taken
$\lambda$ to have components $A \to I \otimes A$, whilst leaving
$\rho$ and $\alpha$ unchanged, we would have obtained (the one-object
case of) Burroni's notion of~\emph{pseudocategory}~\cite{Burroni1971}; 
on the other hand, if we had reversed the sense of $\alpha$ whilst leaving
$\lambda$ and $\rho$ unchanged, we would have obtained something very close to Grandis'
notion~\cite{Grandis2006} of \emph{d-lax $2$-category}.

The second choice concerns the axioms the maps
in~\eqref{eq:constraints} must satisfy. We have already said that Mac\
Lane's five axioms are no longer equivalent to Kelly's two in the skew
setting; so why, then, should we prefer the former to the latter?  For
monoidal categories, we justified the axioms in terms of a theorem
stating that all diagrams of coherence morphisms commute. In the
skew-monoidal case, we have no such justification, since a general
diagram of skew-monoidal coherence morphisms need \emph{not} commute;
describing the structure these coherence morphisms determine is in fact
quite subtle~\cite{Tosmc}.

The objective of the paper is to provide a perspective on
skew-monoidal structure which, amongst other things, makes it quite
apparent why the choices made above are natural ones. 
To do this, we use the Catalan simplicial set $\mathbb C$ mentioned
above. It turns out to be quite easy to describe: it is itself a 
nerve, the nerve of the monoidal poset
$(\mathbf 2, \vee, 0)$. In particular, this makes it
$2$-coskeletal---meaning that for $n > 2$, each $n$-simplex boundary
has a unique filler---and thus completely determined by its $0$-, $1$-
and $2$-simplices, of which it has one, two, and five respectively; more
generally, the number of $n$-simplices is the $n$-th Catalan number.

Our perspective, then, is that $\mathbb C$ classifies skew-monoidal 
structures in the sense that simplicial
maps from $\mathbb C$ into a suitably-defined nerve of $\mathbf{Cat}$
are precisely skew-monoidal categories. More generally,
skew monoidales in a monoidal bicategory $\CK$ are classified by maps
from $\mathbb C$ into the simplicial nerve of $\CK$.

The two non-degenerate $2$-simplices in $\mathbb C$ encode the tensor
and unit operations borne by any skew-monoidal category; the
non-degenerate $3$-simplices encode $\alpha$, $\lambda$ and $\rho$,
with the orientations specified above; whilst the non-degenerate
$4$-simplices encode the skew-monoidal axioms. The coskeletality means
that, in particular, the $3$- and $4$-simplices are completely
determined by the $2$-simplices, and it is in this sense that our
perspective justifies the choices of coherence data and axioms for a
skew-monoidal structure.

There is another well-known connection between the Catalan numbers and (skew) monoidal structure, 
arising from the fact that the $n$th Catalan number $C_n$ is the number of ways to bracket an 
$(n+1)$-fold product. 
The connection described in this paper seems to be quite different; 
indeed, in the context of this paper $C_n$ involves $n$-fold products. 

This work is inspired by an old idea of Michael Johnson on how to capture not
only associativity but also unitality constraints simplicially.  He
reminded us of this in a recent talk~\cite{MJ2013} to the
Australian Category Seminar.

\section{Skew-monoidal categories and skew monoidales}\label{sec:skew}

As in the introduction, a \emph{skew-monoidal category} is a  
 category $\CA$ equipped with a unit element $I \in \CA$, a tensor
product $\otimes \colon \CA \times \CA \to \CA$, and natural families
of (non-invertible) constraint maps $\alpha$, $\lambda$ and $\rho$ as
in~\eqref{eq:constraints}, all subject to the commutativity of the
following diagrams---wherein tensor is denoted by juxtaposition---for
all $a,b,c,d \in \CA$:

\begin{equation}\label{axiom1}
\cd[@C-36pt]{
& &(ab)(cd) \ar[drr]^{\alpha}& & \\
((ab)c)d \ar[urr]^{\alpha} \ar[dr]_{\alpha1} & & & & a(b(c(d)) \\
& (a(bc))d \ar[rr]_{\alpha} && a((bc)d) \ar[ur]_{1\alpha}&
}
\end{equation}
\begin{equation}\label{axiom4}
\cd[@!C@C-6pt]{
& (aI)b \ar[r]^{\alpha} & a(Ib) \ar[dr]^{1\lambda} & \\
ab \ar[ur]^{\rho1} \ar[rrr]_{\mathrm{id}} & & & ab
}
\end{equation}
\begin{equation}\label{axiom3}
\cd[@!C@C-12pt]{
& I(ab) \ar[dr]^{\lambda} & \\
(Ia)b \ar[ur]^{\alpha} \ar[rr]_{\lambda1} &  & ab
}
\end{equation}
\begin{equation}\label{axiom2}
\cd[@!C@C-12pt]{
& (ab)I \ar[dr]^{\alpha} & \\
ab \ar[ur]^{\rho} \ar[rr]_{1\rho} &  & a(bI)
}
\end{equation}
\begin{equation}\label{axiom5}
\cd[@!C@C-6pt]{
& II \ar[dr]^{\lambda} & \\
I \ar[ur]^{\rho} \ar[rr]_{\mathrm{id}} && I\rlap{ .}
}
\end{equation}
We may also say that $(\otimes,I,\alpha,\lambda,\rho)$ is a
\emph{skew-monoidal structure} on $\CA$.  What we call skew-monoidal
is what Szlach\'anyi~\cite{Szl2012} calls left-monoidal, whilst what he calls
right-monoidal we would call \emph{opskew-monoidal}; an
opskew-monoidal structure on $\CA$ is the same as a skew-monoidal
structure on $\CA^\mathrm{op}$.

More generally, we can consider skew monoidales in a monoidal
bicategory. A monoidal bicategory is a one-object tricategory in the
sense of~\cite{GPS}; it thus comprises a bicategory $\CB$ equipped with a
unit object $I$ and tensor product homomorphism $\otimes \colon \CB
\times \CB \to \CB$ which is associative and unital only up to
pseudonatural equivalences $\mathfrak a$, $\mathfrak l$ and $\mathfrak
r$. The coherence of these equivalences is witnessed by invertible
modifications $\pi, \mu, \sigma$ and $\tau$, whose components are
$2$-cells with boundaries those of the
axioms~\eqref{axiom1}--\eqref{axiom2} above, and an invertible
$2$-cell $\theta$ whose boundary is that of \eqref{axiom5}. The
modifications $\pi, \mu, \sigma$ and $\tau$ are as in~\cite{GPS},
though we write $\sigma$ and $\tau$ for what there are called
$\lambda$ and $\rho$; whilst $\theta \colon \mathfrak r_I \circ \mathfrak l_I \Rightarrow 1_I \colon I \to I$ can be defined from the
remaining coherence data as the composite:
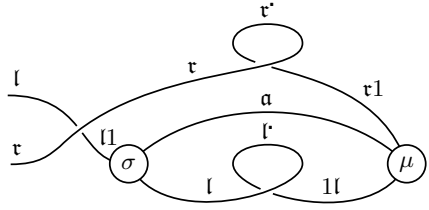
\begin{equation*}
\begin{tikzpicture}[y=0.65pt, x=0.65pt,yscale=-1, inner sep=0pt, outer sep=0pt, every text node part/.style={font=\scriptsize} ]
  \path[use as bounding box] (-30, 930) rectangle (235,1045);  \path[draw=black,line join=miter,line cap=butt,line width=0.650pt]
  (40.0000,1012.3622) .. controls (60.0000,1052.3622) and (140.0000,1029.0471) ..
  node[above right=0.1cm,pos=0.37] {$\mathfrak l$}(140.0000,1012.3622) .. controls (140.0000,997.2450) and (100.0000,997.8136) ..
  node[above=0.07cm] {$\mathfrak l^\centerdot$}  (100.0000,1012.3622) .. controls (100.0000,1017.3926) and (106.3215,1022.7556) ..
  (115.8549,1026.9783)(122.9680,1029.7141) .. controls (148.2131,1038.1193) and (186.9261,1038.5100) ..
  node[above left=0.1cm,pos=0.5] {$1 \mathfrak l$}(200.0000,1012.3622);
  \path[draw=black,line join=miter,line cap=butt,line width=0.650pt]
  (40.0000,1012.3622) .. controls (80.0000,972.3622) and (160.0000,972.3622) ..
  node[above=0.1cm,pos=0.495] {$\mathfrak a$} (200.0000,1012.3622);
  \path[draw=black,line join=miter,line cap=butt,line width=0.650pt]
  (-29.6956,972.6465) .. controls (-7.5878,972.6465) and (2.5465,981.2164) ..
  node[above right=0.12cm,at start] {\!$\mathfrak l$}(10.0677,990.3944)
  (13.1931,994.3876) .. controls (20.2119,1003.6506) and (25.9496,1012.3622) ..
  node[above right=0.055cm,pos=0.35] {$\mathfrak l 1$}(40.0000,1012.3622)
  (200.0000,1012.3622) .. controls (186.7856,972.7190) and (147.3782,963.6344) ..
  node[above right=0.07cm,pos=0.43] {$\mathfrak r 1$}(122.1590,956.2645)
  (112.8208,953.3027) .. controls (105.0036,950.4894) and (100.0000,947.3640) ..
  (100.0000,942.3622) .. controls (100.0000,927.5293) and (140.0000,927.5293) ..
  node[above=0.1cm] {$\mathfrak r^\centerdot$}(140.0000,942.3622) .. controls (140.0000,963.1400) and (50.2642,952.4640) ..
  node[above=0.1cm,pos=0.57] {$\mathfrak r$} (0.2778,1002.3622) .. controls (-5.7407,1007.0671) and (-10.8148,1012.6400) ..
  node[above right=0.12cm,at end]{\!\!$\mathfrak r$}(-28.0000,1012.6400);
  \path[fill=black] (40,1012.3622) node[circle, draw, line width=0.65pt, minimum width=5mm, fill=white, inner sep=0.25mm] (text2987) {$\sigma$};
  \path[fill=black] (200,1012.3622) node[circle, draw, line width=0.65pt, minimum width=5mm, fill=white, inner sep=0.25mm] (text2991) {$\mu$};
\end{tikzpicture}
\quad \text.
\end{equation*}

Here, and elsewhere in this paper, we use string notation to display
composite $2$-cells in a bicategory, with objects represented by
regions, $1$-cells by strings, and generating $2$-cells by
vertices. We orient our string diagrams with 1-cells proceeding down
the page and 2-cells proceeding from left to right.  If a 1-cell $\xi$
belongs to a specified adjoint equivalence, then we will denote its
specified adjoint pseudoinverse by $\xi^\centerdot$, and as usual with
adjunctions, will draw the unit and counit of the adjoint equivalence
in string diagrams as simple caps and cups. In representing the
monoidal structure of a bicategory, we we notate the tensor product
$\otimes$ by juxtaposition, notate the structural 1-cells $\mathfrak
a, \mathfrak l, \mathfrak r$ and 2-cells $\pi, \nu, \lambda, \rho$
explicitly, and use string crossings to notate pseudonaturality
constraint 2-cells, and also instances of the pseudofunctoriality of
$\otimes$ of the form $(f \otimes 1) \circ (1 \otimes g) \cong (1
\otimes g) \circ (f \otimes 1)$ (the interchange isomorphisms).

With our notational conventions now established, we now define a
\emph{skew monoidale} in a monoidal bicategory $\CB$ to be given by an object
$A \in \CB$, unit and multiplication morphisms $i \colon I \to A$ and $t
\colon A \otimes A \to A$, and (non-invertible) coherence $2$-cells
\begin{equation*} 
 \cd{
    (A \otimes A) \otimes A \rtwocell{drr}{\alpha} \ar[rr]^{\mathfrak a} \ar[d]_{t \otimes A} & &
    A \otimes (A \otimes A) \ar[d]^{A \otimes t} \\
    A \otimes A \ar[r]_t & A & A \otimes A \ar[l]^t
  }
  \ \  \text{and} \ \ 
  \cd{
    I \otimes A \ar[d]_{i \otimes A} \rtwocell{dr}{\lambda}
    \ar[r]^-{\mathfrak l} & 
    A \rtwocell{dr}{\rho} \ar@{=}[d] \ar[r]^-{\mathfrak r}
    & A \otimes I \ar[d]^{A \otimes i} \\
    A \otimes A \ar[r]_-t & A & A \otimes A \ar[l]^-t
  }
\end{equation*}
subject to the following five axioms, the appropriate analogues of~\eqref{axiom1}--\eqref{axiom5}.
\begin{align*}
\vcenter{\hbox{\begin{tikzpicture}[y=0.8pt, x=0.8pt,yscale=-1, inner sep=0pt, outer sep=0pt, every text node part/.style={font=\scriptsize} ]
  \path[draw=black,line join=miter,line cap=butt,line width=0.650pt]
    (0.5761,996.5573) .. controls (25.8299,996.0522) and (35.7165,994.4461) ..
    node[above right=0.12cm,at start] {$\!\!(t  1)  1$}(48.0259,1000.9550);
  \path[draw=black,line join=miter,line cap=butt,line width=0.650pt]
    (0.5761,1020.4270) .. controls (32.7994,1020.6795) and (38.4294,1009.3612) ..
    node[above right=0.12cm,at start] {\!$t  1$}(48.0259,1003.0477);
  \path[draw=black,line join=miter,line cap=butt,line width=0.650pt]
    (49.8969,1000.5044) .. controls (58.2307,980.6554) and (151.0215,979.4784) ..
    node[above right=0.25cm,pos=0.4] {$\mathfrak a  1$}(166.1738,987.8122);
  \path[draw=black,line join=miter,line cap=butt,line width=0.650pt]
    (49.6135,1003.1578) .. controls (60.9777,1006.6933) and (72.9109,1017.9473) ..
    node[below left=0.07cm,pos=0.55]{$t  1$} (80.4466,1027.1123);
  \path[draw=black,line join=miter,line cap=butt,line width=0.650pt]
    (0.5761,1042.0698) .. controls (14.4657,1042.3223) and (75.2981,1040.0698) ..
    node[above right=0.12cm,at start] {\!$t$} (80.8539,1029.7127);
  \path[draw=black,line join=miter,line cap=butt,line width=0.650pt]
    (82.2825,1030.8845) .. controls (90.1821,1054.0933) and (173.6677,1051.1703)
    .. node[above left=0.12cm,at end] {$t$\!}(214.7867,1051.3122);
  \path[draw=black,line join=miter,line cap=butt,line width=0.650pt]
    (141.5590,1019.2256) .. controls (143.0743,1011.6494) and (166.0258,992.5776)
    .. node[below right=0.02cm] {$1  \mathfrak a$}(166.0258,992.5776);
  \path[draw=black,line join=miter,line cap=butt,line width=0.650pt]
    (141.6203,1021.3685) .. controls (141.6203,1021.3685) and (163.4571,1014.0062)
    .. node[above left=0.12cm,at end] {$1  (1  t)\!\!$}(214.7868,1014.7638);
  \path[draw=black,line join=miter,line cap=butt,line width=0.650pt]
    (141.4111,1023.0061) .. controls (141.4111,1023.0061) and (142.5493,1033.5038)
    .. node[above left=0.12cm,at end] {$1  t$\!}(214.7868,1033.7564);
  \path[draw=black,line join=miter,line cap=butt,line width=0.650pt]
    (167.5590,988.4831) .. controls (167.5590,988.4831) and (177.6142,980.4754) ..
    node[above left=0.12cm,at end] {$\mathfrak a$\!} (214.7868,979.9703);
  \path[draw=black,line join=miter,line cap=butt,line width=0.650pt]
    (167.9162,991.3403) .. controls (172.8729,995.3375) and (185.9479,996.5546) ..
    node[above left=0.12cm,at end] {$\mathfrak a$\!}(214.7867,996.5546);
  \path[draw=black,line join=miter,line cap=butt,line width=0.650pt]
    (82.5536,1028.3546) .. controls (96.1906,1037.3648) and (132.5520,1035.8496)
    .. node[below right=0.12cm,pos=0.42] {$1  t$} (139.8756,1022.0412);
  \path[draw=black,line join=miter,line cap=butt,line width=0.650pt]
    (82.7084,1025.8212) .. controls (84.1609,1008.5230) and (138.1030,990.5459) ..
    node[above left=0.07cm,pos=0.6] {$\mathfrak a$}(166.0786,990.2620)
    (49.8145,1001.8228) .. controls (60.2802,1001.0663) and (79.7259,1003.0374) ..
    node[above right=0.08cm,pos=0.4,rotate=-9] {$(1 t)  1$}(98.0839,1006.4270)
    (105.9744,1007.9878) .. controls (121.2679,1011.2244) and (134.6866,1015.3648) ..
    node[below left=0.05cm,pos=0.68,rotate=-18] {$1  (t  1)$}(139.8779,1019.5827);
  \path[fill=black] (166.83624,990.0094) node[circle, draw, line width=0.65pt,
    minimum width=5mm, fill=white, inner sep=0.25mm] (text3313) {$\pi$    };
  \path[fill=black] (48.501251,1001.512) node[circle, draw, line width=0.65pt,
    minimum width=5mm, fill=white, inner sep=0.25mm] (text3305) {$\alpha  1$    };
  \path[fill=black] (81.600487,1027.7435) node[circle, draw, line width=0.65pt,
    minimum width=5mm, fill=white, inner sep=0.25mm] (text3309) {$\alpha$    };
  \path[fill=black] (140.38065,1021.0309) node[circle, draw, line width=0.65pt,
    minimum width=5mm, fill=white, inner sep=0.25mm] (text3317) {$1  \alpha$  };
\end{tikzpicture}}}
\quad\ \  &= \quad\ \
\vcenter{\hbox{\begin{tikzpicture}[y=0.8pt, x=0.7pt,yscale=-1, inner sep=0pt, outer sep=0pt, every text node part/.style={font=\scriptsize} ]
  \path[draw=black,line join=miter,line cap=butt,line width=0.650pt]
    (0.5761,1010.9221) .. controls (22.7994,1011.1746) and (33.5330,1012.5294) ..
    node[above right=0.12cm,at start] {\!$t  1$}(49.1780,1027.2420);
  \path[draw=black,line join=miter,line cap=butt,line width=0.650pt]
    (0.5761,1042.0698) .. controls (14.4657,1042.3223) and (43.6151,1040.3578) ..
    node[above right=0.12cm,at start] {\!$t$}(49.1709,1030.0007);
  \path[draw=black,line join=miter,line cap=butt,line width=0.650pt]
    (51.1755,1030.5965) .. controls (67.4280,1056.6856) and (126.7968,1041.6654)
    .. node[below right=0.14cm] {$t$} (139.8511,1033.1665);
  \path[draw=black,line join=miter,line cap=butt,line width=0.650pt]
    (142.8500,1030.5854) .. controls (142.8500,1030.5854) and (184.5612,1028.3914)
    .. node[above left=0.12cm,at end] {$1  t$\!}(225.8909,1029.1490);
  \path[draw=black,line join=miter,line cap=butt,line width=0.650pt]
    (142.6407,1033.9511) .. controls (142.6407,1033.9511) and (172.9964,1047.0248)
    .. node[above left=0.12cm,at end] {$t$\!}(225.8909,1046.9894);
  \path[draw=black,line join=miter,line cap=butt,line width=0.650pt]
    (51.8895,1026.1092) .. controls (64.3077,997.3630) and (142.1189,975.3128) ..
    node[above left=0.12cm,at end] {$\mathfrak a$\!}(225.8909,975.6050)
    (52.8868,1028.6427) .. controls (64.4224,1023.3508) and (87.4860,1019.3533) ..
    node[below right=0.1cm,pos=0.53] {$1  t$} (113.0660,1016.4776)
    (121.4833,1015.5795) .. controls (135.8648,1014.1237) and (150.7159,1013.0048) ..
    node[above=0.1cm] {$1  t$} (164.5150,1012.1936)
    (173.4370,1011.7035) .. controls (191.3809,1010.7870) and (206.7880,1010.4220) ..
    node[above left=0.12cm,at end] {$1  (1  t)\!\!$} (225.8909,1010.5363)
    (143.3647,1028.4425) .. controls (173.9707,1013.3776) and (183.0558,992.0340) ..
    node[above left=0.12cm,at end] {$\mathfrak a$\!}(225.8909,992.6100)
    (1.1521,981.0038) .. controls (30.9106,980.4472) and (58.1432,987.9204) ..
    node[above right=0.12cm,at start] {$\!\!(t  1)  1$}(81.9942,998.1170)
    (88.1242,1000.8286) .. controls (108.0234,1009.9228) and (125.3927,1020.6737) ..
    node[above right=0.1cm,pos=0.22] {$t  1$}
    (139.6961,1029.7577);
  \path[fill=black] (49.917496,1028.8956) node[circle, draw, line width=0.65pt,
    minimum width=5mm, fill=white, inner sep=0.25mm] (text3309) {$\alpha$    };
  \path[fill=black] (139.03424,1031.976) node[circle, draw, line width=0.65pt,
    minimum width=5mm, fill=white, inner sep=0.25mm] (text3317) {$\alpha$  };
\end{tikzpicture}}}\\[6pt]
\vcenter{\hbox{\begin{tikzpicture}[y=0.8pt, x=0.7pt,yscale=-1, inner sep=0pt, outer sep=0pt, every text node part/.style={font=\scriptsize} ]
  \path[draw=black,line join=miter,line cap=butt,line width=0.650pt] (20.0000,912.3622) .. controls (43.4793,912.3622) and (63.3478,901.7100) .. 
  node[above right=0.12cm,at start] {\!$\tee$}(64.6522,898.8839);
  \path[draw=black,line join=miter,line cap=butt,line width=0.650pt] (67.0808,898.9849) .. controls (78.1440,914.6333) and (153.3281,912.3622) .. 
  node[above left=0.12cm,at end] {$\tee$\!}(180.0000,912.3622);
  \path[draw=black,line join=miter,line cap=butt,line width=0.650pt] (35.0870,872.9554) .. controls (37.6957,884.2598) and (62.6957,891.0578) ..
  node[below left=0.03cm,pos=0.5] {$\tee 1$}(64.4348,894.3187);
  \path[draw=black,line join=miter,line cap=butt,line width=0.650pt] (35.3307,867.7426) .. controls (44.1089,854.4448) and (158.1673,850.3255) .. 
  node[above=0.07cm,pos=0.5] {$\mathfrak r 1$}(166.2328,860.9103);
  \path[draw=black,line join=miter,line cap=butt,line width=0.650pt] (67.9130,896.4926) .. controls (67.9130,896.4926) and (120.2302,898.0607) .. 
    node[above=0.065cm,pos=0.45] {$1t$}(127.6087,893.4491);
  \path[draw=black,line join=miter,line cap=butt,line width=0.650pt] (36.3058,870.5159) .. controls (50.8051,868.6788) and (66.2099,872.3069) .. 
  node[above right=0.03cm,pos=0.5,rotate=-15] {$(1 \eye)1$}(81.0828,876.9756)
  (87.8061,879.1556) .. controls (102.1676,883.9328) and (115.8114,889.0641) .. 
  node[above right=0.03cm,pos=0.22,rotate=-18] {$1 (\eye 1)$}(127.3913,890.4057)
  (67.1072,894.2441) .. controls (74.8719,875.7718) and (111.4903,864.4336) .. 
  node[above left=0.05cm,pos=0.7] {$\mathfrak a$}(164.3696,863.8187);
  \path[draw=black,line join=miter,line cap=butt,line width=0.650pt] (130.6522,892.5796) .. controls (138.3381,891.9647) and (163.9007,873.7454) .. 
  node[below right=0.03cm,pos=0.5] {$1 \mathfrak l$}(165.8990,866.5206);
  \path[fill=black] (65.479271,896.49127) node[circle, draw, line width=0.65pt, minimum width=5mm, fill=white, inner sep=0.25mm] (text4130) {$\alpha$     };
  \path[fill=black] (34.740463,870.51276) node[circle, draw, line width=0.65pt, minimum width=5mm, fill=white, inner sep=0.25mm] (text4134) {$\rho 1$     };
  \path[fill=black] (128.51724,892.49457) node[circle, draw, line width=0.65pt, minimum width=5mm, fill=white, inner sep=0.25mm] (text4138) {$1 \lambda$     };
  \path[fill=black] (166.03291,864.05157) node[circle, draw, line width=0.65pt, minimum width=5mm, fill=white, inner sep=0.25mm] (text4142) {$\mu$   };
\end{tikzpicture}}} \quad \ &= \ \quad 
\vcenter{\hbox{\begin{tikzpicture}[y=0.7pt, x=0.7pt,yscale=-1, inner sep=0pt, outer sep=0pt, every text node part/.style={font=\scriptsize} ]
  \path[draw=black,line join=miter,line cap=butt,line width=0.650pt] (30.0000,1022.3622) -- 
  node[above right=0.12cm,at start] {\!$t$}node[above left=0.12cm,at end] {$t$\!}(150.0000,1022.3622);
\end{tikzpicture}}}\\[6pt]
\vcenter{\hbox{\begin{tikzpicture}[y=0.7pt, x=0.7pt,yscale=-1, inner sep=0pt, outer sep=0pt, every text node part/.style={font=\scriptsize} ]
  \path[draw=black,line join=miter,line cap=butt,line width=0.650pt] (30.0000,992.3622) .. controls (49.9858,992.3622) and (57.9934,997.7506) .. 
  node[above right=0.12cm,at start] {\!$\tee 1$}(60.4529,1001.1324);
  \path[draw=black,line join=miter,line cap=butt,line width=0.650pt] (30.0000,1012.3622) .. controls (50.0425,1012.3622) and (58.4628,1009.5953) .. 
  node[above right=0.12cm,at start] {\!$t$}(60.9223,1006.2135);
  \path[draw=black,line join=miter,line cap=butt,line width=0.650pt] (30.0000,972.3622) .. controls (52.9363,972.3622) and (72.2458,974.7838) .. 
  node[above right=0.12cm,at start] {\!\!$(\eye 1) 1$}(87.1746,980.7689)
  (93.5809,983.6996) .. controls (104.9262,989.6034) and (113.1445,998.1058) .. 
  node[above right=0.03cm,pos=0.38] {$\eye 1$}(117.7983,1009.8696)
  (63.7389,1000.0316) .. controls (69.5802,988.6564) and (102.3454,972.3622) .. 
  node[above left=0.12cm,at end] {$\mathfrak a$\!}(150.0000,972.3622)
  (64.9686,1004.3523) .. controls (78.6374,1006.2162) and (94.1546,1003.1361) .. 
  node[above=0.07cm,pos=0.5] {$1\tee$}(109.2174,999.6320)
  (114.4978,998.3882) .. controls (127.3619,995.3357) and (139.6875,992.3622) .. 
  node[above left=0.12cm,at end] {$1 \tee$\!}(150.0000,992.3622);
  \path[draw=black,line join=miter,line cap=butt,line width=0.650pt] (120.0000,1012.3622) -- 
  node[above left=0.12cm,at end] {$\mathfrak l$\!}(150.0000,1012.3622);
  \path[draw=black,line join=miter,line cap=butt,line width=0.650pt] (63.9470,1007.7837) .. controls (75.0666,1019.6587) and (106.1606,1024.4623) .. 
  node[below=0.05cm,pos=0.5] {$\tee$}(117.7486,1013.9325);
  \path[fill=black] (62.102421,1003.787) node[circle, draw, line width=0.65pt, minimum width=5mm, fill=white, inner sep=0.25mm] (text3566) {$\alpha$
     };
  \path[fill=black] (118.9784,1012.3953) node[circle, draw, line width=0.65pt, minimum width=5mm, fill=white, inner sep=0.25mm] (text3570) {$\lambda$
   };
\end{tikzpicture}}} \quad \ &= \ \quad 
\vcenter{\hbox{\begin{tikzpicture}[y=0.7pt, x=0.7pt,yscale=-1, inner sep=0pt, outer sep=0pt, every text node part/.style={font=\scriptsize} ]
  \path[draw=black,line join=miter,line cap=butt,line width=0.650pt] (30.0000,912.3622) .. controls (50.6075,912.3622) and (57.8287,917.4085) .. 
  node[above right=0.12cm,at start] {\!\!$(\eye 1) 1$}(59.9808,919.8680);
  \path[draw=black,line join=miter,line cap=butt,line width=0.650pt] (30.0000,932.3622) .. controls (47.8314,932.3622) and (57.9934,928.6729) .. 
  node[above right=0.12cm,at start] {\!$\tee 1$}(60.4529,924.6762);
  \path[draw=black,line join=miter,line cap=butt,line width=0.650pt] (103.7058,920.3556) .. controls (106.1653,914.5142) and (134.3086,912.3622) .. 
  node[above left=0.12cm,at end] {$\mathfrak a$\!}(150.0000,912.3622);
  \path[draw=black,line join=miter,line cap=butt,line width=0.650pt] (30.0000,952.3622) .. controls (76.4799,952.3622) and (98.0970,945.8975) ..
  node[above right=0.12cm,at start] {\!$\tee$}(114.2384,940.3189)(119.5946,938.4494) .. controls (129.3527,935.0470) and (137.7583,932.3622) ..
  node[above left=0.12cm,at end] {$1 \tee$\!}(150.0000,932.3622)(103.7058,925.3077) .. controls (109.5471,929.3043) and (120.4844,952.3622) .. 
  node[above left=0.12cm,at end] {$\mathfrak l$\!}(150.0000,952.3622);
  \path[draw=black,line join=miter,line cap=butt,line width=0.650pt] (62.2980,922.3160) -- 
  node[above=0.09cm,pos=0.5] {$\mathfrak l 1$}(100.3701,922.3160);
  \path[fill=black] (60.869564,922.14478) node[circle, draw, line width=0.65pt, minimum width=5mm, fill=white, inner sep=0.25mm] (text3884) {$\lambda 1$
     };
  \path[fill=black] (102.13043,922.14478) node[circle, draw, line width=0.65pt, minimum width=5mm, fill=white, inner sep=0.25mm] (text3888) {$\sigma$
   };
\end{tikzpicture}}}\\[6pt]
\vcenter{\hbox{\begin{tikzpicture}[y=0.7pt, x=0.7pt,yscale=-1, inner sep=0pt, outer sep=0pt, every text node part/.style={font=\scriptsize} ]
  \path[draw=black,line join=miter,line cap=butt,line width=0.650pt] 
  (110.4694,936.4878) .. controls (114.7736,934.3357) and (120.0000,922.3622) .. 
  node[above left=0.12cm,at end] {$\mathfrak a$\!}(150.0000,922.3622);
  \path[draw=black,line join=miter,line cap=butt,line width=0.650pt] 
  (45.6001,918.6729) .. controls (63.5753,918.6729) and (75.6670,917.2210) .. 
  node[above=0.07cm,pos=0.45] {$1i$}(85.1137,915.1836)
  (90.1837,913.9779) .. controls (108.7191,909.1594) and (117.7227,902.3622) .. 
  node[above left=0.12cm,at end] {$1 i$\!}(150.0000,902.3622)
  (45.9075,915.2580) .. controls (54.5158,896.5043) and (101.7078,882.3622) .. 
  node[above left=0.12cm,at end] {$\mathfrak r$\!}(150.0000,882.3622)
  (30.0000,882.3622) .. controls (49.1430,882.2684) and (61.6720,886.9527) .. 
  node[above right=0.12cm,at start] {\!$t$}(70.5131,893.6222)
  (74.8228,897.2828) .. controls (80.1942,902.3978) and (83.9960,908.2524) .. 
  (87.2443,913.8769) .. controls (93.7170,925.0845) and (97.9916,935.3785) .. 
  node[above right=0.055cm,pos=0.45] {$\tee 1$}(108.1057,937.0861);
  \path[draw=black,line join=miter,line cap=butt,line width=0.650pt] 
  (110.7769,939.5952) .. controls (116.0033,940.8250) and (132.7835,942.3622) ..
  node[above left=0.12cm,at end] {$1 \tee$\!}(150.0000,942.3622);
  \path[draw=black,line join=miter,line cap=butt,line width=0.650pt] 
  (110.4694,942.8316) .. controls (113.5438,951.7473) and (131.8377,962.3622) ..
  node[above left=0.12cm,at end] {$\tee$\!}(150.0000,962.3622);
  \path[draw=black,line join=miter,line cap=butt,line width=0.650pt] 
  (46.1157,921.3937) .. controls (51.6495,932.1540) and (83.9305,950.6003) .. 
    node[above right=0.055cm,pos=0.45] {$\tee$}(108.5255,941.3772);
  \path[fill=black] (42.733845,918.9342) node[circle, draw, line width=0.65pt, minimum width=5mm, fill=white, inner sep=0.25mm] (text3236) {$\rho$ };
  \path[fill=black] (109.44783,939.83997) node[circle, draw, line width=0.65pt, minimum width=5mm, fill=white, inner sep=0.25mm] (text3240) {$\alpha$ };
\end{tikzpicture}}} \quad \ &= \ \quad 
\vcenter{\hbox{\begin{tikzpicture}[y=0.7pt, x=0.7pt,yscale=-1, inner sep=0pt, outer sep=0pt, every text node part/.style={font=\scriptsize} ]
  \path[draw=black,line join=miter,line cap=butt,line width=0.650pt] (30.0000,1022.3622) -- 
  node[above right=0.12cm,at start] {\!$t$}node[above left=0.12cm,at end] {$t$\!}(150.0000,1022.3622);
  \path[draw=black,line join=miter,line cap=butt,line width=0.650pt] (150.0000,942.3622) .. controls (126.0199,942.3622) and (109.2231,948.0580) .. 
  node[above left=0.12cm,at start] {$\mathfrak r$\!}(106.1488,953.5919);
  \path[draw=black,line join=miter,line cap=butt,line width=0.650pt] (150.0000,982.3622) .. controls (124.4568,981.7473) and (110.0000,967.8961) ..
  node[above left=0.12cm,at start] {$\mathfrak a$\!}(105.6959,957.7506)
  (150.0000,962.3622) .. controls (139.9401,962.2395) and (129.7698,965.4711) .. 
  node[above left=0.12cm,at start] {$1 \eye$\!}(118.1710,969.3606)
  (112.6296,971.2200) .. controls (96.3251,976.6597) and (77.0420,982.5314) .. 
  node[below=0.05cm,pos=0.4] {$1(1 \eye)$}(51.5372,982.2002);
  \path[draw=black,line join=miter,line cap=butt,line width=0.650pt] (150.0000,1002.3622) .. controls (111.2519,1002.3622) and (62.6049,999.7407) .. 
  node[above left=0.12cm,at start] {$1 \tee$\!}(51.2298,985.2911);
  \path[draw=black,line join=miter,line cap=butt,line width=0.650pt] (103.2991,956.1342) .. controls (87.6198,956.4416) and (57.1834,965.0499) .. 
    node[above left=0.055cm,pos=0.45] {$1 \mathfrak r$}(51.3421,978.5771);
  \path[fill=black] (48.575161,981.65149) node[circle, draw, line width=0.65pt, minimum width=5mm, fill=white, inner sep=0.25mm] (text3352) {$1\rho$
     };
  \path[fill=black] (104.83627,955.82672) node[circle, draw, line width=0.65pt, minimum width=5mm, fill=white, inner sep=0.25mm] (text3356) {$\tau$
   };
\end{tikzpicture}}}\\[6pt]
\vcenter{\hbox{\begin{tikzpicture}[y=0.7pt, x=0.7pt,yscale=-1, inner sep=0pt, outer sep=0pt, every text node part/.style={font=\scriptsize} ]
  \path[draw=black,line join=miter,line cap=butt,line width=0.650pt] 
  (34.0000,882.3622) .. controls (53.5806,882.2663) and (65.0541,887.1694) .. 
  node[above right=0.12cm,at start] {\!$\eye$}(72.8455,894.0826)
  (76.6782,897.9452) .. controls (81.0329,902.8997) and (84.1354,908.4937) .. 
  node[above right=0.04cm,pos=0.77] {$\eye 1$}(87.2443,913.8769) .. controls (93.7170,925.0845) and (97.9916,935.3785) .. (108.1057,937.0861)
  (45.9075,915.2580) .. controls (54.5158,896.5043) and (126.9177,877.7506) .. 
  node[above=0.07cm,pos=0.5] {$\mathfrak r$}(140.1620,884.8217)
  (45.6001,918.6729) .. controls (62.8702,918.6729) and (75.6953,918.1544) .. 
  node[above=0.09cm,pos=0.5] {$1 \eye$}(86.3824,917.4146)
  (91.8786,917.0018) .. controls (105.3036,915.9160) and (115.6132,914.5265) .. 
  node[above=0.09cm,pos=0.59] {$1 \eye$}(128.3871,913.5523)
  (133.7731,913.1782) .. controls (141.7041,912.6822) and (150.7402,912.3622) .. 
  node[above left=0.12cm,at end] {$\eye$\!}(162.0000,912.3622)
  (111.0843,939.2547) .. controls (131.3752,934.0283) and (133.2198,896.5374) .. 
  node[below right=0.05cm,pos=0.34] {$\mathfrak l$}(139.5471,888.5440);
  \path[draw=black,line join=miter,line cap=butt,line width=0.650pt] (46.1157,921.3937) .. controls (51.6495,932.1540) and (83.9305,950.6003) .. 
  node[below left=0.05cm,pos=0.59] {$\tee$}(108.5255,941.3772);
  \path[fill=black] (42.733845,918.9342) node[circle, draw, line width=0.65pt, minimum width=5mm, fill=white, inner sep=0.25mm] (text3236) {$\rho$
     };
  \path[fill=black] (109.44783,939.83997) node[circle, draw, line width=0.65pt, minimum width=5mm, fill=white, inner sep=0.25mm] (text3240) {$\lambda$
     };
  \path[fill=black] (141.11392,887.26813) node[circle, draw, line width=0.65pt, minimum width=5mm, fill=white, inner sep=0.25mm] (text4834) {$\theta$
   };
\end{tikzpicture}}} \quad \ &= \ \quad 
\vcenter{\hbox{\begin{tikzpicture}[y=0.7pt, x=0.7pt,yscale=-1, inner sep=0pt, outer sep=0pt, every text node part/.style={font=\scriptsize} ]
  \path[draw=black,line join=miter,line cap=butt,line width=0.650pt] (30.0000,1022.3622) -- 
  node[above right=0.12cm,at start] {\!$\eye$}node[above left=0.12cm,at end] {$\eye$\!}(150.0000,1022.3622);\end{tikzpicture}}}
\rlap{ .}\end{align*}


Note that a skew monoidale in the monoidal
bicategory $(\mathbf{Cat}, \times, 1)$ is precisely a skew-monoidal
category, whilst a skew monoidale in the $2$-cell dual
$(\mathbf{Cat}^\mathrm{co}, \times, 1)$ is an opskew-monoidal category.

\section{Nerves of monoidal bicategories}\label{sec:nerves}
Before describing the simplicial set $\mathbb C$ that
classifies skew-monoidal categories, and more generally,
skew monoidales in a monoidal bicategory, we will first describe the
nerve construction by which we will assign a simplicial set $\mathrm N
\CB$ to a given monoidal bicategory $\CB$; the classification of
skew monoidales in $\CB$ will then be in terms of simplicial maps
$\mathbb C \to \mathrm N\CB$.

First let us recall some basic definitions.  We write $\Delta$
for the simplicial category; the objects are $[n] = \{0,\dots,n\}$ for
$n\ge 0$ and the morphisms are order-preserving functions.  Objects
$X$ of $\mathrm{SSet} = [\Delta^{\mathrm{op}}, \mathrm{Set}]$ are
called {\em simplicial sets}; we write $X_n$ for $X([n])$ and call its
elements \emph{$n$-simplices} of $X$. We use the notation $d_{i}
\colon X_n \to X_{n-1}$ and $s_i \colon X_n \to X_{n+1}$ for the face
and degeneracy maps, induced by acting on $X$ by the maps $\delta_i
\colon [n-1]\to [n]$ and $\sigma_{i} \colon [n+1]\to [n]$ of $\Delta$,
the respective injections and surjections for which
$\delta_{i}^{-1}(i) = \emptyset$ and $\sigma_i^{-1}(i) = \{i, i+1\}$.
An $(n+1)$-simplex $x$ is called {\em degenerate} when it is in the image
of some $s_i$, and \emph{non-degenerate} otherwise.

A simplicial set is called \emph{$r$-coskeletal} when it lies in the
image of the right Kan extension functor $[(\Delta^{(r)})^\mathrm{op},
\mathrm{Set}] \to [\Delta^\mathrm{op}, \mathrm{Set}]$, where
$\Delta^{(r)} \subset \Delta$ is the full subcategory on those $[n]$
with $n\le r$.  In elementary terms, a simplicial set is
$r$-coskeletal when every $n$-boundary with $n > r$ has a unique
filler; here, an \emph{$n$-boundary} in a simplicial set is a
collection of $(n-1)$-simplices $(x_0, \dots, x_n)$ satisfying $d_j(x_i)
 = d_i(x_{j+1})$ for all $0 \leqslant i \leqslant j < n$;
a \emph{filler} for such a boundary is an $n$-simplex $x$ with $d_i(x)
 = x_i$ for $i = 0, \dots, n$.

As we noted above, a monoidal bicategory is a one-object tricategory
in the sense of~\cite{GPS}. There are several known constructions of
nerves for tricategories; the one of interest to us is essentially
Street's $\omega$-categorical nerve~\cite{aos}, restricted from dimension
$\omega$ to dimension $3$, and generalised from strict to weak
$3$-categories. An explicit description of this nerve is given
in~\cite{GRoT}; we now reproduce the details for the case of a monoidal
bicategory $\CB$. For such a $\CB$, the \emph{nerve} $\mathrm N\CB$ is the simplicial
set defined as follows:
\begin{itemize}[nolistsep]
\item There is a unique $0$-simplex, denoted $\star$.
\item A $1$-simplex is an object $A_{01}$ of $\CB$; its two faces are
  necessarily $\star$.
\item A $2$-simplex is given by objects $A_{12}, A_{02}, A_{01}$ of
  $\CB$ together with a $1$-cell ${A}_{012} \colon {A}_{12} \otimes {A}_{01} \to
  {A}_{02}$; its three faces are $A_{12}$, $A_{02}$, and $A_{01}$.
\item A $3$-simplex is given by:
\begin{itemize}[noitemsep]
\item Objects ${A}_{ij}$ for each $0 \leqslant i < j
  \leqslant 3$;
\item $1$-cells ${A}_{ijk} \colon {A}_{jk} \otimes {A}_{ij} \to {A}_{ik}$ for
  each $0 \leqslant i < j < k \leqslant 3$;
\item A $2$-cell
  \begin{equation*}
 \cd{
    ({A}_{23} \otimes {A}_{12}) \otimes {A}_{01} \rtwocell{drr}{{A}_{0123}} \ar[rr]^{\mathfrak a} \ar[d]_{{A}_{123} \otimes 1} & &
    {A}_{23} \otimes ({A}_{12} \otimes {A}_{01}) \ar[d]^{1 \otimes {A}_{012}} \\
    {A}_{13} \otimes {A}_{01} \ar[r]_{{A}_{013}} & {A}_{03} & {A}_{23}
    \otimes {A}_{02}\rlap{ ;} \ar[l]^{{A}_{023}}
  }
  \end{equation*}
\end{itemize}
its four faces are $A_{123}$, $A_{023}$, $A_{013}$ and $A_{012}$.
\item A $4$-simplex is given by:
\begin{itemize}[noitemsep]
\item Objects ${A}_{ij}$ for each $0 \leqslant i < j
  \leqslant 4$;
\item  $1$-cells ${A}_{ijk} \colon {A}_{jk} \otimes {A}_{ij} \to {A}_{ik}$ for
  each $0 \leqslant i < j < k \leqslant 4$;
\item $2$-cells ${A}_{ijk\ell} \colon {A}_{ij\ell} \circ ({A}_{jk\ell}
  \otimes 1) \Rightarrow {A}_{ik\ell} \circ (1 \otimes {A}_{ijk}) \circ
  \mathfrak a$ for each $0 \leqslant i < j < k < \ell \leqslant 4$
\end{itemize}
such that the $2$-cell equality
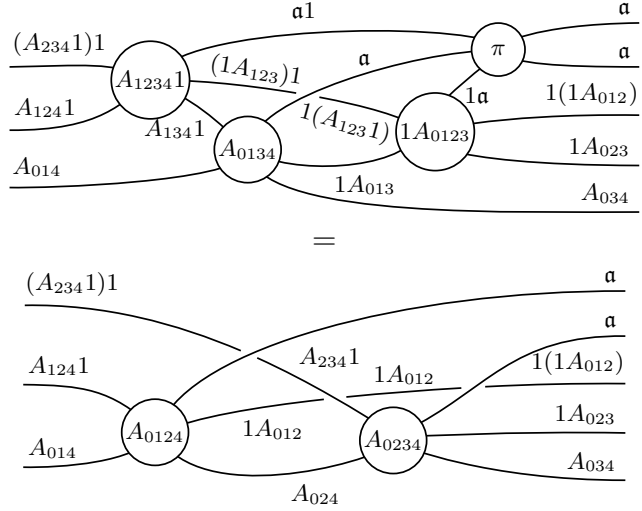
\begin{gather*}
\vcenter{\hbox{\begin{tikzpicture}[y=1pt, x=1.1pt,yscale=-1, inner sep=0pt, outer sep=0pt, every text node part/.style={font=\scriptsize} ]
  \path[draw=black,line join=miter,line cap=butt,line width=0.650pt]
    (0.5761,996.5573) .. controls (25.8299,996.0522) and (35.7165,994.4461) ..
    node[above right=0.15cm,at start] {$\!\!({A}_{234}  1)  1$}(48.0259,1000.9550);
  \path[draw=black,line join=miter,line cap=butt,line width=0.650pt]
    (0.5761,1020.4270) .. controls (32.7994,1020.6795) and (38.4294,1009.3612) ..
    node[above right=0.17cm,at start] {$\!\!{A}_{124}  1$}(48.0259,1003.0477);
  \path[draw=black,line join=miter,line cap=butt,line width=0.650pt]
    (49.8969,1000.5044) .. controls (58.2307,980.6554) and (151.0215,979.4784) ..
    node[above right=0.25cm,pos=0.4] {$\mathfrak a  1$}(166.1738,987.8122);
  \path[draw=black,line join=miter,line cap=butt,line width=0.650pt]
    (49.6135,1003.1578) .. controls (60.9777,1006.6933) and (72.9109,1017.9473) ..
    node[below left=0.05cm,pos=0.57]{${A}_{134}  1$} (80.4466,1027.1123);
  \path[draw=black,line join=miter,line cap=butt,line width=0.650pt]
    (0.5761,1042.0698) .. controls (14.4657,1042.3223) and (75.2981,1040.0698) ..
    node[above right=0.15cm,at start] {$\!\!{A}_{014}$} (80.8539,1029.7127);
  \path[draw=black,line join=miter,line cap=butt,line width=0.650pt]
    (82.2825,1030.8845) .. controls (90.1821,1054.0933) and (173.6677,1051.1703)
    .. node[above left=0.12cm,at end] {${A}_{034}$}(214.7867,1051.3122);
  \path[draw=black,line join=miter,line cap=butt,line width=0.650pt]
    (141.5590,1019.2256) .. controls (143.0743,1011.6494) and (166.0258,992.5776)
    .. node[below right=0.04cm] {$1  \mathfrak a$}(166.0258,992.5776);
  \path[draw=black,line join=miter,line cap=butt,line width=0.650pt]
    (141.6203,1021.3685) .. controls (141.6203,1021.3685) and (163.4571,1014.0062)
    .. node[above left=0.12cm,at end] {$1  (1  {A}_{012})\!\!$}(214.7868,1014.7638);
  \path[draw=black,line join=miter,line cap=butt,line width=0.650pt]
    (141.4111,1023.0061) .. controls (141.4111,1023.0061) and (142.5493,1033.5038)
    .. node[above left=0.12cm,at end] {$1  {A}_{023}$}(214.7868,1033.7564);
  \path[draw=black,line join=miter,line cap=butt,line width=0.650pt]
    (167.5590,988.4831) .. controls (167.5590,988.4831) and (177.6142,980.4754) ..
    node[above left=0.12cm,at end] {$\mathfrak a$} (214.7868,979.9703);
  \path[draw=black,line join=miter,line cap=butt,line width=0.650pt]
    (167.9162,991.3403) .. controls (172.8729,995.3375) and (185.9479,996.5546) ..
    node[above left=0.12cm,at end] {$\mathfrak a$}(214.7867,996.5546);
  \path[draw=black,line join=miter,line cap=butt,line width=0.650pt]
    (82.5536,1028.3546) .. controls (96.1906,1037.3648) and (132.5520,1035.8496)
    .. node[below right=0.12cm,pos=0.42] {$1  {A}_{013}$} (139.8756,1022.0412);
  \path[draw=black,line join=miter,line cap=butt,line width=0.650pt]
    (82.7084,1025.8212) .. controls (84.1609,1008.5230) and (138.1030,990.5459) ..
    node[above left=0.07cm,pos=0.6] {$\mathfrak a$}(166.0786,990.2620)
    (49.8145,1001.8228) .. controls (60.2802,1001.0663) and (79.7259,1003.0374) ..
    node[above right=0.08cm,pos=0.4,rotate=-9] {$(1  {A}_{123})  1$}(98.0839,1006.4270)
    (105.9744,1007.9878) .. controls (121.2679,1011.2244) and (134.6866,1015.3648) ..
    node[below left=0.08cm,pos=0.7,rotate=-15] {$1  ({A}_{123}  1)$}(139.8779,1019.5827);
  \path[fill=black] (166.83624,990.0094) node[circle, draw, line width=0.65pt,
    minimum width=5mm, fill=white, inner sep=0.25mm] (text3313) {$\ \ 
      \pi\ \ $    };
  \path[fill=black] (48.501251,1001.512) node[circle, draw, line width=0.65pt,
    minimum width=5mm, fill=white, inner sep=0.25mm] (text3305) {${A}_{1234}  1$    };
  \path[fill=black] (81.600487,1027.7435) node[circle, draw, line width=0.65pt,
    minimum width=5mm, fill=white, inner sep=0.25mm] (text3309) {${A}_{0134}$    };
  \path[fill=black] (145.38065,1021.0309) node[circle, draw, line width=0.65pt,
    minimum width=5mm, fill=white, inner sep=0.25mm] (text3317) {$1  {A}_{0123}$  };
\end{tikzpicture}}} \\ = \\
\vcenter{\hbox{\begin{tikzpicture}[y=1pt, x=1pt,yscale=-1, inner sep=0pt, outer sep=0pt, every text node part/.style={font=\scriptsize} ]
  \path[draw=black,line join=miter,line cap=butt,line width=0.650pt]
    (0.5761,1010.9221) .. controls (22.7994,1011.1746) and (33.5330,1012.5294) ..
    node[above right=0.12cm,at start] {\!${A}_{124}  1$}(49.1780,1027.2420);
  \path[draw=black,line join=miter,line cap=butt,line width=0.650pt]
    (0.5761,1042.0698) .. controls (14.4657,1042.3223) and (43.6151,1040.3578) ..
    node[above right=0.12cm,at start] {\!${A}_{014}$}(49.1709,1030.0007);
  \path[draw=black,line join=miter,line cap=butt,line width=0.650pt]
    (51.1755,1030.5965) .. controls (67.4280,1056.6856) and (126.7968,1041.6654)
    .. node[below right=0.14cm] {${A}_{024}$} (139.8511,1033.1665);
  \path[draw=black,line join=miter,line cap=butt,line width=0.650pt]
    (142.8500,1030.5854) .. controls (142.8500,1030.5854) and (184.5612,1028.3914)
    .. node[above left=0.12cm,at end] {$1  {A}_{023}$}(225.8909,1029.1490);
  \path[draw=black,line join=miter,line cap=butt,line width=0.650pt]
    (142.6407,1033.9511) .. controls (142.6407,1033.9511) and (172.9964,1047.0248)
    .. node[above left=0.12cm,at end] {${A}_{034}$}(225.8909,1046.9894);
  \path[draw=black,line join=miter,line cap=butt,line width=0.650pt]
    (51.8895,1026.1092) .. controls (64.3077,997.3630) and (142.1189,975.3128) ..
    node[above left=0.12cm,at end] {$\mathfrak a$}(225.8909,975.6050)
    (52.8868,1028.6427) .. controls (64.4224,1023.3508) and (87.4860,1019.3533) ..
    node[below right=0.1cm,pos=0.53] {$1  {A}_{012}$} (113.0660,1016.4776)
    (121.4833,1015.5795) .. controls (135.8648,1014.1237) and (150.7159,1013.0048) ..
    node[above=0.1cm] {$1  {A}_{012}$} (164.5150,1012.1936)
    (173.4370,1011.7035) .. controls (191.3809,1010.7870) and (206.7880,1010.4220) ..
    node[above left=0.12cm,at end] {$1  (1  {A}_{012})\!\!$} (225.8909,1010.5363)
    (143.3647,1028.4425) .. controls (173.9707,1013.3776) and (183.0558,992.0340) ..
    node[above left=0.12cm,at end] {$\mathfrak a$}(225.8909,992.6100)
    (1.1521,981.0038) .. controls (30.9106,980.4472) and (58.1432,987.9204) ..
    node[above right=0.12cm,at start] {$\!\!({A}_{234}  1)  1$}(81.9942,998.1170)
    (88.1242,1000.8286) .. controls (108.0234,1009.9228) and (125.3927,1020.6737) ..
    node[above right=0.1cm,pos=0.22] {${A}_{234}  1$}
    (139.6961,1029.7577);
  \path[fill=black] (49.917496,1028.8956) node[circle, draw, line width=0.65pt,
    minimum width=5mm, fill=white, inner sep=0.25mm] (text3309) {${A}_{0124}$    };
  \path[fill=black] (139.03424,1031.976) node[circle, draw, line width=0.65pt,
    minimum width=5mm, fill=white, inner sep=0.25mm] (text3317) {${A}_{0234}$  };
\end{tikzpicture}}}
\end{gather*}
holds. The five faces of this simplex are $A_{1234}$, $A_{0234}$,
$A_{0134}$, $A_{0124}$ and $A_{0123}$.
\item Higher-dimensional simplices are determined by the requirement
  that $\mathrm{N}\CB$ be $4$-coskeletal.
\end{itemize}
It remains to describe the degeneracy operators. The degeneracy of the
unique $0$-simplex is the unit object $I \in \CB$; the two
degeneracies $s_0(A), s_1(A)$ of a
$1$-simplex $A \in \CB$ are the unit constraints $\mathfrak
r^\centerdot \colon A \otimes I \to A$ and $\mathfrak l \colon I
\otimes A \to A$; the three degeneracies $s_0(\gamma), s_1(\gamma)$ and $s_2(\gamma)$ of a $2$-simplex
$\gamma \colon B \otimes C \to A$ are the respective $2$-cells
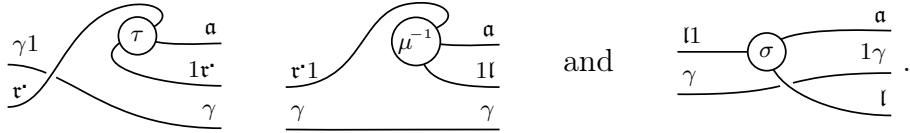
\begin{equation*}
\vcenter{\hbox{\begin{tikzpicture}[y=0.8pt, x=0.8pt,yscale=-1, inner sep=0pt, outer sep=0pt, every text node part/.style={font=\scriptsize} ]
  \path[draw=black,line join=miter,line cap=butt,line width=0.650pt] 
  (0.0000,902.3622) .. controls (22.2057,902.3622) and (24.4296,862.1105) .. 
  node[above right=0.12cm,at start] {$\!\mathfrak r^\centerdot$}(54.6378,854.8261) .. controls (70.4902,851.0034) and (80.9687,860.6367) .. (63.1815,866.8469)
  (0.0000,882.3622) .. controls (6.5599,882.4187) and (12.3710,883.9537) .. 
  node[above right=0.12cm,at start] {$\!\gamma 1$}(18.0721,886.3023)(22.7007,888.3801) .. controls (40.8765,897.1386) and (59.6979,912.3622) .. 
  node[above left=0.12cm,at end] {$\gamma$\!}(100.0000,912.3622);
  \path[draw=black,line join=miter,line cap=butt,line width=0.650pt] (59.1409,868.1591) .. controls (29.0888,882.3013) and (70.2005,892.3622) .. 
  node[above left=0.12cm,at end] {$1 \mathfrak r^\centerdot$\!}(100.0000,892.3622);
  \path[draw=black,line join=miter,line cap=butt,line width=0.650pt] (62.6764,869.3723) .. controls (70.2525,874.1706) and (91.3989,872.3622) .. 
  node[above left=0.12cm,at end] {$\mathfrak a$\!}(100.0000,872.3622);
  \path[fill=black] (60.710678,868.80756) node[circle, draw, line width=0.65pt, minimum width=5mm, fill=white, inner sep=0.25mm] (text14966) {$\tau$   };
 \end{tikzpicture}}}
\qquad
\vcenter{\hbox{\begin{tikzpicture}[y=0.8pt, x=0.8pt,yscale=-1, inner sep=0pt, outer sep=0pt, every text node part/.style={font=\scriptsize} ]
  \path[draw=black,line join=miter,line cap=butt,line width=0.650pt] (0.0000,892.3622) .. controls (36.8530,892.1096) and (33.3645,857.8413) .. 
  node[above right=0.12cm,at start] {\!$\mathfrak r^\centerdot 1$}(56.4056,853.3108) .. controls (75.1348,849.6282) and (85.9633,861.3740) .. (63.1815,866.8469);
  \path[draw=black,line join=miter,line cap=butt,line width=0.650pt] (0.0000,912.3622) .. controls (29.3000,912.6147) and (43.6628,912.3622) .. 
  node[above left=0.12cm,at end] {$\gamma$\!}node[above right=0.12cm,at start] {\!$\gamma$}(100.0000,912.3622);
  \path[draw=black,line join=miter,line cap=butt,line width=0.650pt] (62.5254,872.1096) .. controls (63.3845,886.3127) and (70.2005,892.3622) .. 
  node[above left=0.12cm,at end] {$1 \mathfrak l$\!}(100.0000,892.3622);
  \path[draw=black,line join=miter,line cap=butt,line width=0.650pt] (62.6764,869.3723) .. controls (70.2525,874.1706) and (91.3989,872.3622) .. 
  node[above left=0.12cm,at end] {$\mathfrak a$\!}(100.0000,872.3622);
  \path[fill=black] (60.963215,870.82788) node[circle, draw, line width=0.65pt, minimum width=5mm, fill=white, inner sep=0.25mm] (text14966) {$\mu^{\!\scriptscriptstyle -1}$
   };
 \end{tikzpicture}}}\qquad\text{and}\qquad
\vcenter{\hbox{\begin{tikzpicture}[y=0.8pt, x=0.8pt,yscale=-1, inner sep=0pt, outer sep=0pt, every text node part/.style={font=\scriptsize} ]
  \path[draw=black,line join=miter,line cap=butt,line width=0.650pt]
  (43.0799,884.7365) .. controls (43.9391,898.9396) and (70.2005,912.3622) .. 
  node[above left=0.12cm,at end] {$\mathfrak l$\!}(100.0000,912.3622)
  (0.0000,902.3622) .. controls (26.8739,902.5938) and (38.5571,900.6887) ..
  node[above right=0.12cm,at start] {\!$\gamma$}(48.0003,898.4810)
  (54.0398,897.0009) .. controls (63.6207,894.6257) and (73.8763,892.3622) .. 
  node[above left=0.12cm,at end] {$1 \gamma$\!}(100.0000,892.3622);
  \path[draw=black,line join=miter,line cap=butt,line width=0.650pt] (43.2310,880.4840) .. controls (49.3612,869.8661) and (91.3989,872.3622) .. 
  node[above left=0.12cm,at end] {$\mathfrak a$\!}(100.0000,872.3622);
  \path[draw=black,line join=miter,line cap=butt,line width=0.650pt] (0.0000,882.3622) -- 
  node[above right=0.12cm,at start] {\!$\mathfrak l 1$}(40.0000,882.3622);
  \path[fill=black] (41.767765,882.36218) node[circle, draw, line width=0.65pt, minimum width=5mm, fill=white, inner sep=0.25mm] (text14966) {$\sigma$
     };
 \end{tikzpicture}}}\rlap{ .}
\end{equation*}
The four degeneracies of a $3$-simplex are
simply the assertions of certain $2$-cell equalities; that these hold
is a consequence of the axioms for a monoidal bicategory. Higher
degeneracies are determined by coskeletality.

\section{The Catalan simplicial set}\label{Cssan}

As mentioned in the introduction, the \emph{Catalan simplicial set} is the nerve of the monoidal category $(\mathbf 2, \vee, 0)$.
The category $\mathbf 2$ has two objects $0$, $1$ and a single morphism $0\to 1$. The tensor for the monoidal structure is disjunction and the unit is $0$.
This is actually a strict monoidal category and so can be regarded as a one-object 2-category $\Sigma \mathbf 2$. 
The Catalan simplicial set $\mathbb C$ is the nerve of this 2-category; its $n$-simplices are normal lax functors $[n] \to \Sigma \mathbf 2$ as in~\cite{aos}.
In section \ref{sec:CatalanClassifies} we will show that $\mathbb C$ classifies skew monoidales in a monoidal bicategory. 

Now consider item (f{}f{}f{}) in Stanley's list~\cite{StanleyEC2} of Catalan sets. These are relations $R \subseteq [n] \times [n]$ that are reflexive, symmetric and have the interpolation property:
\begin{equation*}
(i,k) \text{ and } i\le j \le k \text{ implies } (i,j) \text{ and } (j,k)\ .
\end{equation*}
Let $\mathbb{K}_n$ be the set of such relations. Each $\mathbb K_n$ has $C_n$ elements.

\begin{proposition}\label{catalanbyrelations}
The assignment sending each normal lax functor $F \colon [n] \to \Sigma \mathbf 2$ to 
$$ \{(i,j)\ :\ F(i\le j) = 0\ \vee\ F(j\le i) = 0\}\ $$
is a bijection $N\Sigma\mathbf 2_n \cong \mathbb{K}_n$. Thus $|\mathbb C_n| = C_n$.
\end{proposition}  
\begin{proof} 
It is easy to see that the relation above is symmetric and reflexive. Suppose that $i\leq j \leq k$ and $(i,k)$ is in the set above. Then we find that $F(j\le k).F(i\le j) \Rightarrow F(i\le k)$ has codomain 0 and so must be the identity on 0. Then the domain is 0 and thus $F(j\le k) = F(i\le j) = 0$.
Now notice that since $\Sigma \mathbf 2$ has a single object and is locally posetal, normal lax functors $[n] \to \Sigma \mathbf 2$ are completely determined by their action on 1-cells. 
The assignment is a bijection because each $F$ is completely determined by which maps $i\le j$ go to 0.
$\square$
\end{proof}

\begin{corollary} The number of non-degenerate simplices of each dimension in 
$\mathbb C$ is given by the Motzkin sequence 
$$1, 1, 2, 4, 9, 21, 51, 127, 323, 835, 2188, 5798, 15511, 41835, \dots \ .$$
\end{corollary} 
\begin{proof}
The Catalan numbers can be obtained from the Motzkin numbers~\cite{oeis} by taking
cardinalities in 
$$\mathbb C_n \cong \displaystyle\sum\limits_{m=0}^{n} \binom{n}{m} \times \mathrm{nd}_m\mathbb C$$
where $\mathrm{nd}_m\mathbb C$ is the set of non-degenerate m-simplices in $\mathbb C$.
$\square$
\end{proof}

The Catalan simplicial set can also be described using ideals.
The category of ordered sets and order-preserving functions is denoted by $\mathrm{Ord}$.
For ordered sets $M$ and $N$, an {\em ideal} $A : M\to N$ is a subset
$A\subseteq N\times M$ (that is, a relation from $M$ to $N$) such that 
$$q\le j, (j,i)\in A, i\le p \text{\ implies\ } (q,p)\in A \ .$$
Composition of ideals is composition of relations.
We have a 2-category $\mathrm{Idl}$ of ordered sets, ideals and inclusions (for example, see~\cite{CS}).
The identity ideal $1_M$ of $M$ is $\{(j,i):j\le i\}$.
Each order-preserving function $\xi : M \to N$ gives rise to ideals
$\xi_* : M \to N$ and $\xi^* : N \to M$ defined by
$$\xi_* = \{(j,i): j \le \xi(i)\} \quad \text{and} \quad \xi^* = \{(i,j): \xi(i)\le j\} \ .$$
Then $1_M\le \xi^* \xi_*$ and $\xi_* \xi^* \le 1_N$. 
This means $\xi_* \dashv \xi^*$ in $\mathrm{Idl}$. 
This defines functors $(-)_* : \mathrm{Ord} \to \mathrm{Idl}$ and
$(-)^* : \mathrm{Ord}^{\mathrm{op}} \to \mathrm{Idl}$, both the identity
on objects. This is all familiar $\CV$-category theory with $\CV = \mathbf 2$.

Let us put
$$\mathbb{S}_{n} = \{B\in \mathrm{Idl}([n],[n]) : 1_{[n]} \le B \} \ .$$
Then each $\xi : [m] \to [n]$ gives a function $\mathbb{S}_{\xi} : \mathbb{S}_{n} \to \mathbb{S}_{m}$
defined by 
$$\mathbb{S}_{\xi}(B) = {\xi}^* B {\xi}_* = \xi^{-1}(B) = \{(p,q) : (\xi(p),\xi(q)) \in B \} \ .$$  
Notice that $1_{[n]} \le B$ implies $1_{[n]} \le {\xi}^* {\xi}_* \le {\xi}^* B {\xi}_* $.  
Thus we have defined a simplicial set 
$\mathbb{S} : \Delta^{\mathrm{op}} \to \mathrm{Set}$,
indeed, a simplicial ordered set $\mathbb{S} : \Delta^{\mathrm{op}} \to \mathrm{Ord}$.  

There is an isomorphism $\mathbb S \cong \mathbb C$. These ideals can
be enumerated using Young diagrams (of which there are a Catalan
number); this provides an alternative proof that there are a Catalan
number of simplices at each dimension of $\mathbb C$.

\begin{remark}
Each Tamari lattice~\cite{TamariBk} has a Catalan number of elements. It is not too hard to define the Tamari order on each $\mathbb C_n$. It is natural to ask whether the simplicial structure on $\mathbb C$ preserves the Tamari order. This is not the case: we have $\rho \le s_1(\eye)$ in $\mathbb C_3$ but $d_1(\rho) = s_1(c) \nleq \eye = d_1(s_1(\eye))$.
\end{remark}

\section{The Catalan simplicial set classifies \\
skew monoidales}\label{sec:CatalanClassifies}

Once we have enumerated the simplices of $\mathbb C$ in low
dimensions, it will become clear that the image of every map
$F\colon\mathbb C \to N\CB$ picks out essentially the data and axioms
for a skew monoidale in $\CB$. More precisely, such maps are in
bijection with skew monoidales in $\CB'$ (defined below). First, it is
important to recognise that, since $N\CB$ is 4-coskeletal, every
simplicial map $F\colon\mathbb C \to N\CB$ is completely determined by
its image on the 4-truncation of its domain. In fact, if two such
simplicial maps are equal on their 3-truncation, then they are equal.

We now investigate the non-degenerate n-simplices in $\mathbb C$ for
$n\leq4$. It is convenient to note that all simplices above dimension
2 are uniquely determined by their faces. As such, every $n$-simplex
$a$ can be identified with the $(n+1)$-tuple $(d_0(a),d_2(a), \dots,
d_{n}(a))$ for $n \geq 2$.  \renewcommand{\aa}{a}
\renewcommand{\ll}{\ell} \newcommand{\rr}{r} \newcommand{\kk}{k}

\begin{itemize}
\item There is a single 0-simplex, call it $\star$.
\item There is a single non-degenerate 1-simplex $c$ whose two faces are necessarily $\star$.
\item There are two non-degenerate 2-simplices:
\begin{align*}
\tee &= (c,c,c) \\
\eye &= (s_0(\star),c, s_0(\star))
\end{align*}
\item There are four non-degenerate 3-simplices:
\begin{align*}
\aa &= (\tee,\tee,\tee,\tee) \\
\ll &= (\eye,s_1(c),\tee,s_1(c)) \\
\rr &= (s_0(c),\tee,s_0(c),\eye) \\
\kk &= (\eye, s_1(c), s_0(c), \eye) 
\end{align*}
\item There are nine non-degenerate 4-simplices:
\begin{equation*}
\begin{tabular}{ l l l }
$A1 = (\aa,\aa,\aa,\aa,\aa)$ \hspace{0.7cm} &
$A6 = (s_0(\eye),\rr, \kk, \ll, s_2(\eye))$ \\
$A2 = (\rr,s_1(\tee),\aa,s_1(\tee),\ll)$ &
$A7 = (\kk,\rr, s_0s_1(c), \ll, \kk)$  \\
$A3 = (\rr,\rr,s_2(\tee),\aa,s_2(\tee))$ &
$A8 = (\ll, s_1(\tee), s_0(\tee), \ll, \kk)$ \\
$A4 = (s_0(\tee),\aa,s_0(\tee),\ll,\ll)$ \hspace{0.7cm} &
$A9 = (\kk, \rr, s_2(\tee), s_1(\tee), \rr)$ \\
$A5 = (s_1(\eye),s_2(\eye), \kk, s_0(\eye), s_1(\eye))$ &
\end{tabular}
\end{equation*}
\end{itemize}
The image of $F$ on the 4-truncation of $\mathbb C$ consists of the following.
\begin{itemize}[nolistsep]
\item A single object $F(c) = A$.
\item Two 1-cells
\begin{equation*}
\cd[]{A\otimes A \ar[r]^-{F(\tee)} & A}\quad\text{and}\quad\cd[]{I\otimes I \ar[r]^-{F(\eye)} & A}\ .
\end{equation*}
\item Four 2-cells
\begin{equation*} 
 \cd{
    (A \otimes A) \otimes A \rtwocell{drr}{F(\aa)} \ar[rr]^{\mathfrak a} \ar[d]_{F(\tee) \otimes A} & &
    A \otimes (A \otimes A) \ar[d]^{A \otimes F(\tee)} \\
    A \otimes A \ar[r]_{F(\tee)} & A & A
    \otimes A \ar[l]^{F(\tee)}
  }
\end{equation*}
\begin{equation*} 
 \cd{
    (A \otimes I) \otimes I \rtwocell{drr}{F(\rr)} \ar[rr]^{\mathfrak a} \ar[d]_{\mathfrak r^\centerdot \otimes 1} & &
    A \otimes (I \otimes I) \ar[d]^{A \otimes F(\eye)} \\
    A \otimes I \ar[r]_{\mathfrak r^\centerdot} & A & A
    \otimes A \ar[l]^{F(\tee)}
  }
\end{equation*}
\begin{equation*} 
 \cd{
    (I \otimes I) \otimes A \rtwocell{drr}{F(\ll)} \ar[rr]^{\mathfrak a} \ar[d]_{F(\eye) \otimes 1} & &
    I \otimes (I \otimes A) \ar[d]^{1 \otimes \mathfrak l} \\
    A \otimes A \ar[r]_{F(\tee)} & A & I
    \otimes A \ar[l]^{\mathfrak l}
  }
\end{equation*}
\begin{equation*} 
 \cd{
    (I \otimes I) \otimes I \rtwocell{drr}{F(\kk)} \ar[rr]^{\mathfrak a} \ar[d]_{F(\eye) \otimes 1} & &
    I \otimes (I \otimes I) \ar[d]^{1 \otimes F(\eye)} \\
    A \otimes I \ar[r]_{\mathfrak r^\centerdot} & A & I
    \otimes A \ar[l]^{\mathfrak l}
  }
\end{equation*}
\item And nine equalities:
\end{itemize}
\begin{equation}\label{stringaxiom1}
        \def\abc{\tee}\def\abd{\tee}\def\abe{\tee}
        \def\acd{\tee}\def\ace{\tee}\def\ade{\tee}
        \def\bcd{\tee}\def\bce{\tee}
        \def\bde{\tee}
        \def\cde{\tee}
        \def\abde{F\aa}\def\acde{F\aa}
        \def\abcd{F\aa}\def\abce{F\aa}
        \def\bcde{F\aa}
        \stringpent
\end{equation}
\begin{equation}\label{stringaxiom2}
        \def\abc{\mathfrak l}\def\abd{\tee}\def\abe{\tee}
        \def\acd{\mathfrak l}\def\ace{\tee}\def\ade{\tee}
        \def\bcd{\eye}\def\bce{\mathfrak r^\centerdot}
        \def\bde{\tee}
        \def\cde{\mathfrak r^\centerdot}
        \def\abde{F\aa}\def\acde{s_1(\tee)}
        \def\abcd{F\ll}\def\abce{s_1(\tee)}
        \def\bcde{F\rr}
        \stringpent
\end{equation}
\begin{equation}\label{stringaxiom3}
        \def\cde{\tee}\def\bde{\mathfrak r^\centerdot}\def\ade{\tee}
        \def\bce{\mathfrak r^\centerdot}\def\ace{\tee}\def\abe{\mathfrak r^\centerdot}
        \def\bcd{\mathfrak r^\centerdot}\def\acd{\tee}
        \def\abd{\mathfrak r^\centerdot}
        \def\abc{\eye}
        \def\abde{s_2\tee}\def\abce{F\rr}
        \def\bcde{s_2\tee}\def\acde{F\aa}
        \def\abcd{F\rr}
        \stringpent
\end{equation}
\begin{equation}\label{stringaxiom4}
        \def\cde{\eye}\def\bde{\mathfrak l}\def\ade{\mathfrak l}
        \def\bce{\tee}\def\ace{\tee}\def\abe{\tee}
        \def\bcd{\mathfrak l}\def\acd{\mathfrak l}
        \def\abd{\tee}
        \def\abc{\tee}
        \def\abde{s_0\tee}\def\abce{F\aa}
        \def\bcde{F\ll}\def\acde{F\ll}
        \def\abcd{s_0\tee}
        \stringpent
\end{equation}
\begin{equation}\label{stringaxiom5}
        \def\cde{\mathfrak l}\def\bde{\eye}\def\ade{\mathfrak l}
        \def\bce{\eye}\def\ace{\eye}\def\abe{\mathfrak l}
        \def\bcd{\mathfrak l}\def\acd{\eye}
        \def\abd{\eye}
        \def\abc{\mathfrak l}
        \def\abde{F\kk}\def\abce{s_2(\eye)}
        \def\bcde{s_1\eye}\def\acde{s_0(\eye)}
        \def\abcd{s_1\eye}
        \stringpent
\end{equation}
\begin{equation}\label{stringaxiom6}
        \def\cde{\eye}\def\bde{\tee}\def\ade{\mathfrak l}
        \def\bce{\mathfrak r^\centerdot}\def\ace{\tee}\def\abe{\mathfrak r^\centerdot}
        \def\bcd{\mathfrak r^\centerdot}\def\acd{\mathfrak l}
        \def\abd{\eye}
        \def\abc{\eye}
        \def\abde{F\kk}\def\abce{F\rr}
        \def\bcde{s_2\eye}\def\acde{F\ll}
        \def\abcd{s_0\eye}
        \stringpent
\end{equation}
\begin{equation}\label{stringaxiom7}
        \def\cde{\eye}\def\bde{\mathfrak l}\def\ade{\mathfrak l}
        \def\bce{\mathfrak r^\centerdot}\def\ace{\tee}\def\abe{\mathfrak r^\centerdot}
        \def\bcd{\eye}\def\acd{\mathfrak l}
        \def\abd{\mathfrak r^\centerdot}
        \def\abc{\eye}
        \def\abde{s_0s_1c}\def\abce{F\rr}
        \def\bcde{F\kk}\def\acde{F\ll}
        \def\abcd{F\kk}
        \stringpent
\end{equation}
\begin{equation}\label{stringaxiom8}
        \def\cde{\eye}\def\bde{\mathfrak l}\def\ade{\mathfrak l}
        \def\bce{\mathfrak r^\centerdot}\def\ace{\tee}\def\abe{\tee}
        \def\bcd{\mathfrak l}\def\acd{\mathfrak l}
        \def\abd{\tee}
        \def\abc{\mathfrak l}
        \def\abde{s_0\tee}\def\abce{s_1\tee}
        \def\bcde{F\kk}\def\acde{F\ll}
        \def\abcd{F\ll}
        \stringpent
\end{equation}
\begin{equation}\label{stringaxiom9}
        \def\cde{\mathfrak r^\centerdot}\def\bde{\tee}\def\ade{t}
        \def\bce{\mathfrak r^\centerdot}\def\ace{\tee}\def\abe{\mathfrak r^\centerdot}
        \def\bcd{\eye}\def\acd{\mathfrak l}
        \def\abd{\mathfrak r^\centerdot}
        \def\abc{\eye}
        \def\abde{s_2\tee}\def\abce{F\rr}
        \def\bcde{F\rr}\def\acde{s_1\tee}
        \def\abcd{F\kk}
        \stringpent
\end{equation}

The similarity with skew monoidales in $\CB$ is immediately clear. 
There is however one problem: the unit map for a skew monoidale is of the form $I \to A$ but $F(\eye)$ is a map $I\otimes I \to F(c)$. 
Similarly, the left and right unit constraints for a skew monoidale have different domains and codomains than $F(\rr)$ and $F(\ll)$. 
This disparity means that the structure given by $F$ is not strictly that of a skew monoidale.
However, the difference amounts to the fact that $I\otimes I$ does not equal $I$.

We address this problem by considering the monoidal bicategory $\CB'$ with the same tensor as $\CB$ but whose unit object is $I\otimes I$. The unit maps for this monoidal structure are
\begin{equation*}
\cd[]
{
 (I\otimes I)\otimes A \ar[r]^-{\mathfrak l_I \otimes A} & I\otimes A \ar[r]^-{\mathfrak l_A} & A
}
\quad\text{and}\quad
\cd[]
{
 A \otimes (I\otimes I) \ar[r]^-{A\otimes \mathfrak r^\centerdot_I} & A\otimes I \ar[r]^-{\mathfrak r_A} & A
}\ .
\end{equation*}
The invertible modifications $\pi, \mu, \sigma$ and $\tau$ are also altered accordingly.
In this case the identity functor on $\CB$ becomes a strong monoidal functor $\CB \to \CB'$.
Having modified the unit object in $\CB$ we can construct a bijection between simplicial maps $F\colon \mathbb C \to N\CB$ and skew monoidales in $\CB'$.

\begin{remark}
In the fundamental example $\CB = \mathrm{Cat}$, there is a further bijection between skew monoidales in $\mathrm{Cat}'$ and skew monoidales in $\mathrm{Cat}$; that is, skew-monoidal categories.
\end{remark}

The first thing to notice is that the equality in \eqref{stringaxiom5}
together with the monoidal bicategory axioms force $F(\kk)$ to be
equal to the 2-cell
\begin{equation} \label{kappastring}
  \vcenter{\hbox{\begin{tikzpicture}[y=0.8pt, x=0.9pt,yscale=-1, inner
        sep=0pt, outer sep=0pt, every text node
        part/.style={font=\scriptsize} ]
  \path[draw=black,line join=miter,line cap=butt,line width=0.650pt] 
  (60.5482,860.1807) .. controls (97.7329,834.6731) and (20.2790,833.4806) .. 
  node[above=0.07cm,pos=0.5] {$\mathfrak r^\centerdot 1$}(21.6649,878.0631) .. controls (22.2692,897.5009) and (50.8157,910.5973) .. (77.9793,909.6016)
  (83.8468,909.1607) .. controls (91.1546,908.3220) and (98.2018,906.3945) .. 
  node[above left=0.055cm,pos=0.55] {$\mathfrak r$}(104.3947,903.2213)
  (10.0000,912.3622) .. controls (28.6404,912.3622) and (47.9703,914.2230) .. 
  node[above right=0.12cm,at start] {\!\!$Fi.1$}(68.7323,916.3393)
  (77.0089,917.1848) .. controls (95.0409,919.0225) and (114.1847,920.8922) .. 
  node[below=0.074cm,pos=0.5] {$Fi$}(134.9003,921.7992)
  (140.5333,922.0201) .. controls (146.8711,922.2390) and (153.3558,922.3622) .. 
  node[above left=0.12cm,at end] {$1.Fi$\!\!}(160.0000,922.3622)
  (10.0000,932.3622) .. controls (46.3655,932.3622) and (117.0711,908.0690) .. 
  node[above right=0.12cm,at start] {\!$\mathfrak r^\centerdot$}(60.5546,892.8673)
  (61.6193,889.4751) .. controls (137.9573,830.7177) and (122.5843,941.9527) .. 
  node[above left=0.12cm,at end] {$\mathfrak l$\!}(160.0000,942.3622)
  (59.5990,867.2517) .. controls (61.2888,870.8982) and (66.6120,874.7013) .. 
  node[above right=0.05cm,pos=0.5,rotate=-35]{$1 \mathfrak l$}(73.0425,878.5617)
  (78.0172,881.4569) .. controls (88.8715,887.6317) and (100.9341,893.8888) .. 
  node[above right=0.035cm,pos=0.4]{$\mathfrak l$}(104.0457,899.8292);
  \path[draw=black,line join=miter,line cap=butt,line width=0.650pt] 
  (62.3769,864.7264) .. controls (128.2265,830.2570) and (130.2744,900.0073) .. 
  node[above left=0.12cm,at end] {$\mathfrak a$\!}(160.0000,902.3622);
  \path[fill=black] (52.527935,864.72632) node[circle, draw, line width=0.65pt, minimum width=5mm, fill=white, inner sep=0.25mm] (text3790) {$\mu^{\scriptscriptstyle -1}$  };
  \path[fill=black] (55.101635,890.83112) node[circle, draw, line width=0.65pt, minimum width=5mm, fill=white, inner sep=0.25mm] (text3790-1) {$\theta^{\scriptscriptstyle -1}$   };
  \path[fill=black] (104.09403,901.18524) node[circle, draw, line width=0.65pt, minimum width=5mm, fill=white, inner sep=0.25mm] (text3790-6) {$\theta$};
      \end{tikzpicture}}}
\end{equation}
and thus completely specified by the coherence data of $\CB$.

We now compare the data comprising the image of $F$ with the data
for a skew monoidale in $\CB'$. At dimensions 0 and 1 these data are
exactly equal: a single object $F(c) = A$ together with two 1-cells
$F(\tee) \colon A \otimes A \to A$ and $F(\eye) \colon I \otimes I \to
A$. At dimension two, the $2$-cell $F(\aa)$ has the same form as the
associativity constraint $\alpha$ for a skew monoidale; whilst, as
observed above, $F(\kk)$ is necessarily of the form~\eqref{kappastring}.  On the other
hand, the data $F(\ll)$ and $F(\rr)$ give rise to left and right unit
constraints $\lambda$ and $\rho$ for a skew monoidale in $\CB'$ upon forming the
composites
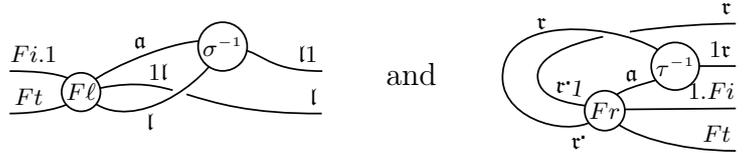
\begin{equation*}
  \vcenter{\hbox{\begin{tikzpicture}[y=0.8pt, x=0.9pt,yscale=-1, inner
        sep=0pt, outer sep=0pt, every text node
        part/.style={font=\scriptsize} ]
  \path[draw=black,line join=miter,line cap=butt,line width=0.650pt] 
  (77.4746,873.6249) .. controls (65.7563,893.8888) and (37.6777,911.3520) .. 
  node[below right=0.06cm,pos=0.6] {$\mathfrak l$}(22.5254,894.9370)
  (119.7970,902.3622) .. controls (89.8639,902.6202) and (75.3129,897.3809) .. 
  node[above left=0.12cm,at start] {$\mathfrak l$\!}(63.5684,893.3012)
  (57.8098,891.3693) .. controls (48.0652,888.3122) and (38.8450,886.9959) .. 
  node[above right=0.06cm,pos=0.4] {$1\mathfrak l$}(21.7678,891.8571);
  \path[draw=black,line join=miter,line cap=butt,line width=0.650pt]
  (77.5761,869.1198) .. controls (65.5753,863.5577) and (30.6214,880.5944) .. 
  node[above left=0.06cm,pos=0.5] {$\mathfrak a$}(22.2728,888.9282);
  \path[draw=black,line join=miter,line cap=butt,line width=0.650pt]
  (120.0495,882.3622) .. controls (94.0270,882.3622) and (96.0094,871.2505) .. 
  node[above left=0.12cm,at start] {$\mathfrak l 1$\!}(80.0495,871.2505);
  \path[draw=black,line join=miter,line cap=butt,line width=0.650pt]
  (-10.0000,882.3622) .. controls (-0.5398,882.3622) and (13.5355,882.8673) .. 
  node[above right=0.12cm,at start] {\!\!$Fi. 1$}(19.0914,888.9282);
  \path[draw=black,line join=miter,line cap=butt,line width=0.650pt]
  (-10.0000,902.3622) .. controls (-0.8054,902.3622) and (14.5457,900.5944) .. 
  node[above right=0.12cm,at start] {\!$Ft$}(18.8388,894.0284);
  \path[fill=black] (78.568542,870.78729) node[circle, draw, line width=0.65pt, minimum width=5mm, fill=white, inner sep=0.25mm] (text3387) {$\sigma^{\scriptscriptstyle -1}$     };
  \path[fill=black] (19.697973,892.25299) node[circle, draw, line width=0.65pt, minimum width=5mm, fill=white, inner sep=0.25mm] (text3391) {$F\ll$   };
      \end{tikzpicture}}} \qquad
\text{and} \qquad
  \vcenter{\hbox{\begin{tikzpicture}[y=0.8pt, x=0.9pt,yscale=-1, inner
        sep=0pt, outer sep=0pt, every text node
        part/.style={font=\scriptsize} ]
  \path[draw=black,line join=miter,line cap=butt,line width=0.650pt] 
  (106.0000,862.3622) -- node[above left=0.12cm,at end] {$1 \mathfrak r$\!}(130.0000,862.3622);
  \path[draw=black,line join=miter,line cap=butt,line width=0.650pt] 
  (77.5241,882.8673) -- node[above left=0.12cm,at end] {$1.Fi$\!\!} (130.0000,882.3622);
  \path[draw=black,line join=miter,line cap=butt,line width=0.650pt] 
  (77.5241,886.7048) .. controls (85.1003,895.0386) and (105.7563,902.3622) .. 
  node[above left=0.12cm,at end] {$Ft$\!}(130.0000,902.3622);
  \path[draw=black,line join=miter,line cap=butt,line width=0.650pt] 
  (103.9848,865.4840) .. controls (99.6916,869.7771) and (81.5698,871.5449) .. 
  node[above left=0.05cm,pos=0.5] {$\mathfrak a$}(77.7817,878.8685);
  \path[draw=black,line join=miter,line cap=butt,line width=0.650pt] 
  (103.7323,860.6857) .. controls (93.8833,844.1683) and (32.0387,832.8929) .. 
  node[above=0.05cm,pos=0.66] {$\mathfrak r$}(33.0291,868.4666) .. controls (33.5742,888.0447) and (62.3769,897.3038) .. 
  node[below=0.05cm,pos=0.8] {$\mathfrak r^\centerdot$}(73.2361,886.1921)
  (72.7259,882.0107) .. controls (66.4734,880.4919) and (47.8665,880.0812) .. 
  node[above right=0.05cm,pos=0.66,rotate=-10] {$\mathfrak r^\centerdot 1$}(47.6192,867.2444) .. controls (47.4461,858.2562) and (59.4891,852.3199) .. (74.8220,848.5084)
  (85.9387,846.1826) .. controls (102.4316,843.2893) and (120.1558,842.3313) .. 
  node[above left=0.12cm,at end] {$\mathfrak r$\!}(130.0000,842.3622);
  \path[fill=black] (75.508904,883.41418) node[circle, draw, line width=0.65pt, minimum width=5mm, fill=white, inner sep=0.25mm] (text3622) {$F\rr$    };
  \path[fill=black] (104.80333,862.45349) node[circle, draw, line width=0.65pt, minimum width=5mm, fill=white, inner sep=0.25mm] (text3626) {$\tau^{\scriptscriptstyle -1}$   };
      \end{tikzpicture}}}\ .
\end{equation*}
The assignments $F(\ll) \mapsto \lambda$ and $F(\rr) \mapsto \rho$ are
in fact bijective, the former since it is given by composing with an
invertible $2$-cell, and the latter since it is given by composition
with an invertible $2$-cell followed by transposition under
adjunction. Thus the two-dimensional data of $F$ and of a skew
monoidale in $\CB'$ are in bijective correspondence.

Finally, we find after some calculation that, with respect to the
$\alpha$, $\lambda$ and $\rho$ defined above, equations
\eqref{stringaxiom1}, \eqref{stringaxiom2}, \eqref{stringaxiom3},
\eqref{stringaxiom4}, \eqref{stringaxiom6} and \eqref{stringaxiom7}
express precisely the five axioms for a skew monoidale in $\CB'$;
equation \eqref{stringaxiom5} specifies $F(\kk)$ and nothing more;
whilst equations \eqref{stringaxiom8} and \eqref{stringaxiom9} are
both equalities which follow using only the axioms for a monoidal
bicategory. We have thus shown:

\begin{theorem}
There is a bijection between simplicial maps $\mathbb C \to N\CB$ and skew monoidales in $\CB'$.
\end{theorem}

\begin{remark}
If we consider a bicategory $\CK$ (not necessarily monoidal), and let $N\CK$ be its nerve, then simplicial maps $\mathbb C \to N\CK$ are monads in $\CK$.
\end{remark}

\begin{remark}
  The assignation $\CB \mapsto N(\CB)$ sending a monoidal bicategory
  to its nerve can be extended to a functor $N \colon
  \mathrm{MonBicat}_s \to \mathrm{SSet}$, where $\mathrm{MonBicat}_s$
  is the category of monoidal bicategories and morphisms which
  strictly preserve all the structure. When seen in this way, the
  nerve functor has a left adjoint $\Phi$; it follows that $\Phi(
  \mathbb{C})$ is the free monoidal bicategory containing a skew
  monoidale.
\end{remark}

\end{document}